\documentclass[pdflatex,iicol,sn-mathphys-num]{sn-jnl}

\usepackage{graphicx}%
\usepackage{multirow}%
\usepackage{amsmath,amssymb,amsfonts}%
\usepackage{amsthm}%
\usepackage{mathrsfs}%
\usepackage[title]{appendix}%
\usepackage{xcolor}%
\usepackage{textcomp}%
\usepackage{manyfoot}%
\usepackage{booktabs}%
\usepackage{algorithm}%
\usepackage{algorithmicx}%
\usepackage{algpseudocode}%
\usepackage{subfigure}%
\usepackage{listings}%
\usepackage{bm}

\usepackage[english]{babel}
\usepackage[round,sort]{natbib}
 
\usepackage{blkarray}
\usepackage{bbm}

\usepackage{amsmath}  
\usepackage{etoolbox} 

\theoremstyle{thmstyleone}%
\newtheorem{theorem}{Theorem}

\newtheorem{definition}{Definition}[section]
\newtheorem{example}{Example}[section]

\newtheorem{lemma}{Lemma}[section]

\newtheorem{proposition}{Proposition}[section]
\newtheorem{remark}{Remark}[section]

\newtheorem{assumption}{Assumption}[section]

\theoremstyle{thmstyletwo}%
\theoremstyle{thmstylethree}%

\DeclareMathOperator{\range}{Range}

\DeclareMathOperator{\tr}{Tr}

\newcommand{\MTc}[1]{\textcolor{red}{}}
\newcommand{\SKc}[1]{\textcolor{red}{}}

\newcommand{\de}{\mathrm{d}}
\usepackage{dsfont}

\newcommand{\diag}{\text{diag}}

\begin{document}

\title{Truss topology design under harmonic loads: Peak power minimization with semidefinite programming}

\author[1]{Shenyuan Ma}\email{shenyma@fel.cvut.cz}

\author[1]{Jakub Mare\v{c}ek, Vyacheslav Kungurtsev}

\author[2]{Marek Tyburec}

\affil[1]{\orgdiv{Czech Technical University in Prague}, \orgname{Faculty of Electrical Engineering, Department of Computer Science}, \orgaddress{Thakurova 7, 16629 Prague 6, \country{Czech Republic}}}

\affil[2]{\orgdiv{Czech Technical University in Prague}, \orgname{Faculty of Civil Engineering, Department of Mechanics}, \orgaddress{Thakurova 7, 16629 Prague 6, \country{Czech Republic}}}
 
\abstract{Designing lightweight yet stiff structures that can withstand vibrations is a crucial task in structural optimization.  Here, we present a novel framework for truss topology optimization under undamped harmonic oscillations. Our approach minimizes the peak power of the structure under harmonic loads, overcoming the limitations of single-frequency and in-phase assumptions found in previous methods. For this, we leverage the concept of semidefinite representable (SDr) functions, demonstrating that while compliance readily conforms to an SDr representation, peak power requires a derivation based on the non-negativity of trigonometric functions.  Finally, we introduce convex relaxations for the minimization problem and provide promising computational results.}

\maketitle

\section*{Acknowledgements}
The authors gratefully acknowledge the support of the Czech Science Foundation (grant number 22-15524S).

\section{Introduction}
The design of lightweight yet stiff vibrating structures has been a long-standing objective in structural optimization \citep{Kang2006}. Here, we consider the optimization of a structural design that is subject to forces in the form of harmonic oscillations, a topic of significance in a range of applications. These include rotating machinery in household appliances, construction machinery for cars and ships \citep{Liu2015rev}, and wind turbine support structures \citep{Liu2015}. In this study, we focus on optimizing truss structures, leveraging their high stiffness-weight ratio \citep{Liu2008,Tyburec2019}.

Structural optimization was pioneered in Michell's seminal work \citep{Michell_1904}, which demonstrated that the optimal material distribution for trusses, under a single loading condition, aligns with the principal stresses. Because the principal stresses are generally not straight, optimum designs may involve an infinite number of bars. To circumvent this, \citet{Dorn1964} discretized the design domain into the so-called ground structure, containing a fixed and finite number of optimized elements.

Relying on ground structure discretization, many theoretical and applied results have been presented in the literature. These include convex formulations for truss topology optimization under single \citep{Dorn1964} or multiple \citep{Achtziger1992, Vandenberghe_1996,Lobo_1998,ben-tal_lectures_2001} loading conditions, or possibly under worst-case loading \citep{BenTal1997}. In dynamics, it has been shown that the constraint for the lower bound on the fundamental free-vibration eigenfrequencies is convex \citep{Ohsaki1999}, also allowing for an efficient maximization of the fundamental free-vibration eigenfrequency \citep{achtziger_structural_2008,Aroztegui2024,nishioka2023}. 

In 2009, \citet{heidari_optimization_2009} observed that, similar to the convex semidefinite programming formulation developed for the static setting \citep{Vandenberghe_1996}, it is also possible to derive a convex semidefinite program for the peak-power minimization under harmonic oscillations. However, the reformulation requires that only a single-frequency in-phase load is present, with the driving frequency strictly below the lowest resonance frequency of the structure itself.

Surveying the more applied literature on continuum topology optimization under harmonic oscillations surprisingly reveals that the same restricting assumptions have been used in this community as well. In particular, the predominant objective is to minimize the dynamic compliance function \citep{Ma1993}, which is physically meaningful only when the driving frequency of the load is below resonance \citep{silva_critical_2019}. Otherwise, the optimization converges to a disconnected material distribution at anti-resonance. To the best of our knowledge, this issue has not yet been resolved.

The second typical assumption of single-frequency loads acting only in phase can be partially justified because it represents the worst-case situation \citep{heidari_optimization_2009}. On the other hand, such designs are hardly optimal for out-of-phase settings, such as unbalanced rotating loads. To the best of our knowledge, the only author who has explored this setting is \citet{Liu2015}, who considered minimization of the energy loss per cycle due to structural damping and minimization of the displacement amplitude. The latter case was handled by aggregating samples over the cycle, with the number of samples balancing computational efficiency with the accuracy of the approximation. 

As was pointed out in \citep{Venini2016}, real-world applications require multiple-frequency loads. In this direction, \citet{Liu2015rev} minimized the integral of displacement amplitudes over a non-uniformly discretized frequency range, whereas \citet{Zhang2015} introduced an aggregation scheme to minimize the worst-case dynamic compliance for multiple frequencies. However, we are not aware of any publication that considers multiple-frequency harmonic loads acting on a structure concurrently.

\subsection{Aims and contributions}
Inspired by the initial results of \citep{heidari_optimization_2009}, in this study, we develop a unifying framework for the minimization of compliance and peak power functions, while avoiding the assumptions for single-frequency and in-phase loads.

Our procedure inherently relies on the notion of semidefinite representable (SDr) functions \citep[Lecture 4.2]{ben-tal_lectures_2001}, which are convex functions whose epigraph can be represented by the projection of a linear matrix inequality (LMI) feasibility set, the so-called LMI shadow. Minimization of such functions under linear constraints can be equivalently reformulated into linear semidefinite programming (SDP), and hence convex, problems. 

We show that for compliance such representation is immediate due to the linearity. For the peak power, however, the SDP representation is not trivial; it relies on the positivity certificate of the trigonometric polynomials. Such certificates have already been extensively studied in the signal processing community for filter design \citep{dumitrescu_positive_2017}. 

Using the notion of a SDr function, we propose convex relaxations of the minimization involving an SDr function and the equilibrium equation as a constraint. Minimization problems subject to an equilibrium equation are, in general, non-convex because of their complementarity nature. In compliance minimization, however, one can exploit the problem structure and eliminate state variables from the minimization in order to obtain a linear SDP problem. Because this is not possible for peak power minimization in general, we propose a convex relaxation with penalization.

For this relaxation, we discuss its link to the Lagrange relaxation of the constrained minimization problem, showing that our convex relaxation corresponds to the Lagrange relaxation with a penalty coefficient equal to $0$. Based on this observation, we adopt a convex relaxation with a non-zero penalty term to generate high-quality sub-optimal solutions.

First, we use the convex relaxation method for single-frequency harmonic loads. The inherent novelty there is allowing the forces to act out of phase. Numerical experiments show that although the relaxation alone is not tight, a penalization term added to the objective function secures convergence to a suboptimal solution that is a feasible point of the original problem. Subsequently, we study the peak power function under loads with multiple harmonics formulated as an SDr.

This article is structured as follows. In Section \ref{sec:notation}, we present the notation and definitions of terms used throughout the paper. In Section \ref{sec:formulation}, we formalize the optimization problem and also present a relaxation procedure to solve the problem in Section \ref{sec:minSDR}. Section \ref{sec:examples} then illustrates our method with five examples. Finally, Section \ref{sec:discussion} summarizes our developments and provides an outlook on related future research.

Throughout the paper, we use italic letter such as $x$ for a real or complex scalar value and boldface letter such as $\mathbf{x}$ for a vector or a matrix of finite dimension.

\section{Notations}\label{sec:notation}
In this section, we introduce the necessary notation. In particular, $\mathbb{R}$ and $\mathbb{C}$ represent real and complex fields, respectively. $\mathbb{T}$ denotes the set of complex numbers of modulus $1$, that is, the set of all complex numbers $z=x + iy$ for which it holds that $\lvert z \rvert = \sqrt{x^2 + y^2}=1$ and $i^2 = -1$. $\mathbb{S}^N$ is the set of $N\times N$ symmetric or Hermitian matrices according to the context. For a complex number $z$, $\overline{z}$ is its complex conjugate, while for a complex vector or matrix, $\mathbf{x}^*=\overline{\mathbf{x}}^T$ stands for the transpose conjugate.

Consider $\mathbf{A}\in\mathbb{R}^{n\times m}$ and assume that $\mathbf A$ has rank $r\leq \min\{n,m\}$. There exists a singular value decomposition $\mathbf{A}=\mathbf{U}\mathbf{\Sigma}\mathbf{V}^T$ where $\mathbf\Sigma$ is $r\times r$ diagonal with positive elements and where $\mathbf{U}\in\mathbb{R}^{n\times r}$ and $\mathbf{V}\in\mathbb{R}^{m\times r}$ whose respective column vectors form an orthonormal family of $\mathbb{R}^n$ (respectively $\mathbb{R}^m$). The Moore-Penrose pseudo inverse $\mathbf{A}^\dagger$ of $\mathbf{A}$, see, e.g., \citep{boyd_convex_2004}, is
\begin{equation}
    \mathbf{A}^\dagger = \mathbf{V}\mathbf{\Sigma}^{-1}\mathbf{U}^T.
\end{equation}
If $m=n$ and $\mathbf{A}$ is invertible then the pseudo-inverse coincides with the usual inverse $\mathbf{A}^\dagger=\mathbf{A}^{-1}$.

We define the Frobenius scalar product $\langle \mathbf{A},\mathbf{B}\rangle=\tr\{\mathbf{A}^T\mathbf{B}\}$, where $\mathbf{A},\mathbf{B}\in\mathbb{S}^N$. When $\mathbf{A},\mathbf{B}$ are real, it holds that $\tr\{\mathbf{A}^T\mathbf{B}\}=\sum_{i,j}A_{i,j}B_{i,j}$, and $\tr\{\mathbf{A}^T\mathbf{B}\}=\sum_{i,j}\overline{A_{i,j}}B_{i,j}$ applies to the Hermitian case. Thus, the set $\mathbb{S}^N$ with the Frobenius scalar product can be identified as a Euclidean vector space.

Let $\mathbb{S}_+^N$ be the set of positive semidefinite (PSD) matrices and define the partial order $\mathbf{A}\succeq 0 \iff \mathbf{A}\in\mathbb{S}_+^N$. A real symmetric (resp. Hermitian) matrix $\mathbf{A}$ is PSD iff $\mathbf{x}^T\mathbf{A}\mathbf{x}\geq 0,\forall \mathbf{x}$ (resp. $\mathbf{x}^*\mathbf{A}\mathbf{x}\geq 0$). 

For a Hermitian matrix $\mathbf{A}$, let $\Re{\mathbf{A}}$ and $\Im{\mathbf{A}}$ denote the matrices obtained by taking the real and imaginary part of $\mathbf{A}$'s entries element-wise. Then, 
\begin{equation}
    \mathbf{A}\succeq 0\iff\begin{pmatrix}\Re{\mathbf{A}}&-\Im{\mathbf{A}}\\\Im{\mathbf{A}}&\Re{\mathbf{A}}\end{pmatrix}\succeq 0.
\end{equation}
The latter matrix is real PSD. Thus, any PSD Hermitian matrix is equivalent to a real symmetric PSD matrix of larger size.

Let $\mathbf{F}$ be an affine map from $\mathbb{R}^n$ (or $\mathbb{C}^n$) to the set of real symmetric matrices $\mathbb{S}^N$, a linear matrix inequality (LMI) is defined as $\mathbf{F}(\mathbf{x})\succeq 0$. The feasible set $\{\mathbf{x} \;\vert\; \mathbf{F}(\mathbf{x})\succeq 0\}$ of the LMI is called a spectrahedron and it is a convex and closed set; see, e.g., \citep{boyd_convex_2004}.

Let $\mathbf{f}:[0,T]\rightarrow \mathbb{R}^{n}$ be a vector valued function from $[0,T]$ to $\mathbb{R}^{n}$, the notation $\dot{\mathbf{f}}$ and $\ddot{\mathbf{f}}$ stand for the first $\frac{\de \mathbf{f}}{\de t}$ and second derivative $\frac{\de ^2\mathbf{f}}{\de t^2}$ of $\mathbf{f}$.

\section{Optimization problem formulation}\label{sec:formulation}

Here, we consider optimization problems occurring in the context of topology optimization of discrete structures with time-varying loads. 

\subsection{Variables of the optimization problem} We can distinguish two types of variables:

\begin{itemize}
    \item The \textit{design variables} $\mathbf{a}\in\mathbb{R}^m$ represent the vector of parameters of the structure, e.g., cross-section areas, (pseudo)densities, etc.
    \item The \textit{state variables} $\mathbf{u}\in\mathbb{R}^n$ describe the physics, e.g., displacements, velocities, temperature, etc.
\end{itemize}

\subsection{Periodic time varying loads}
Let $\mathbf{f}$ represent the load of the structure.
We assume here that the time-varying load $\mathbf{f}$ contains $N$ harmonic components of base angular frequency $\omega$, described by a sequence of complex vector-valued Fourier coefficients $\mathbf{c}_k(\mathbf{f})\in\mathbb{C}^d,\forall k=-N,\dots,N$. $\mathbf{f}$ is thus a periodic function:
\begin{equation}
	\mathbf{f}(t)=\sum_{k=-N}^{N}\mathbf{c}_k(\mathbf{f})e^{ik\omega t}.
\end{equation}
$\mathbf{f}$ is real-valued, implying that the complex Fourier coefficients of $\mathbf{f}$ satisfy the symmetric condition $\mathbf{c}_{-k}(\mathbf{f})=\overline{\mathbf{c}_k(\mathbf{f})}$, where the complex conjugation is taken entry-wise. We also assume that there is no constant component in the load, thus $\mathbf{c}_0(\mathbf{f})$. 

\subsection{Equilibrium equation}
By the finite element discretization, the nodal velocity $\mathbf{v}$ is the solution of the ordinary differential equation (ODE)
\begin{equation}
	\mathbf{M}(\mathbf{a})\ddot{\mathbf{v}}+\mathbf{K(a)v}=\dot{\mathbf{f}}.
\end{equation}

For any design variable $\mathbf{a}$, we can define an ordinary differential operator $\mathbf{L}(\mathbf{a})$ such that for each vector-valued function $\mathbf{v}$, there is:
\begin{equation}
    \mathbf{L}(\mathbf{a}):\mathbf{v}(t)\mapsto \mathbf{M(a)}\ddot{\mathbf{v}}+\mathbf{K(a)v}. 
\end{equation}
The solution of the previous system of ODE can now be consicely written as $\mathbf{L(a)v}=\dot{\mathbf{f}}$, to which we refer by the equilibrium equation.

The steady-state solution of the ODE is a periodic function with $N$ harmonic components as well, satisfying
\begin{equation}
	(-k^2\omega ^2\mathbf{M(a)}+\mathbf{K(a)})\mathbf{c}_k(\mathbf{v})=ik\omega \mathbf{c}_k(\mathbf{f}),\forall k.
\end{equation}
The equilibrium equation relates the state variables and design variables:
\begin{equation}
    \begin{split}
        &\mathbf{L(a)v}=\dot{\mathbf{f}} \\
        \iff& (-k^2\omega ^2\mathbf{M(a)}+\mathbf{K(a)})\mathbf{c}_k(\mathbf{v})=ik\omega \mathbf{c}_k(\mathbf{f}),\forall k.
    \end{split}
\end{equation}%

Using the equilibrium equation as a constraint in the minimization \eqref{eqn:general_minimization}, we automatically eliminate the values of $\mathbf{a}$ for which the design does not allow to carry the load, i.e. $\mathbf{f}\notin\range{\mathbf{L(a)}}$.

The matrix $\mathbf{M(a)}$ and $\mathbf{K(a)}$ are called the mass and the stiffness matrix. They depend linearly on the design variables $a$ in the current study,
\begin{equation}
	\begin{split}
		\mathbf{M(a)}=\sum_{i=1}^m a_i\mathbf{M}_i,\\
		\mathbf{K(a)}=\sum_{i=1}^m a_i\mathbf{K}_i,
	\end{split}
\end{equation}
where $\mathbf{M}_i$ and $\mathbf{K}_i$ are PSD mass and stiffness matrices of individual finite elements. 

We assume that the highest harmonic frequency $N\omega$ is less than or equal to the smallest non-singular free-vibration angular eigenfrequency. It is shown in \citep{achtziger_structural_2008} that this is equivalent to
\begin{equation}
    -N^2\omega^2\mathbf{M(a)}+\mathbf{K(a)}\succeq 0.
\end{equation}

For the sake of conciseness, we use the notation of dynamic stiffness matrix $\mathbf{K}_\lambda(\mathbf{a}) = -\lambda^2\mathbf{M}(a) + \mathbf{K(a)}$ for any real number $\lambda$.

\subsection{Further constraints on design variables} 
With $\mathbf{a}\in\mathbb{R}^m$ representing the vector of the cross-section areas of individual truss elements, we collect a number of constraints on the design variables $\mathbf{a}$:
\begin{itemize}
	\item $a_i\geq 0, \forall i=1,\dots,m$, securing non-negative cross-section areas,
	\item $m-\mathbf{q}^T\mathbf{a}\ge 0$, providing an upper bound for structural mass, with $\mathbf{q}\in\mathbb{R}^m$ being the constant vector of the element weight contributions,
	\item\label{constraint_fv} $\mathbf{K}_\lambda(\mathbf{a})\succeq 0$, ensuring that the smallest non-singular free-vibration angular eigenfrequency is at least $\lambda$, see, e.g., \citep{achtziger_structural_2008}.
\end{itemize}

\subsection{Objective function} In this study, we consider the peak power $p[\mathbf{u}]$ as an objective function. It is the maximum value of the instant power delivered to the structure by the load $\mathbf{f}(t)$, i.e.,
\begin{equation}
	p[\mathbf{c(v)}]=\underset{t\in\left[0,\frac{2\pi}{\omega }\right]}{\max}\{\left\lvert \mathbf{f}(t)^T\mathbf{v}(t)\right\lvert\}.
\end{equation}

\subsection{Peak power minimization} To summarize, we aim at solving the following minimization problem:
\begin{equation}\label{eqn:peak_power_min}
	\begin{split}
		\underset{\mathbf{a},\mathbf{c(v)}}{\min} ~&~ p[\mathbf{c(v)}]\\
		\text{ s.t } & \left\{\begin{array}{l}
			a_i\geq0, \mathbf{a}^T\mathbf{q}\leq m,\\
            \mathbf{K}_{N\omega}(\mathbf{a})\succeq 0\\
			\mathbf{K}_{k\omega}(\mathbf{a})\mathbf{c}_k(\mathbf{v})=ik\omega \mathbf{c}_k(\mathbf{f}),\forall k
		\end{array}\right.
	\end{split}
\end{equation}
This is a particular instance of a broader class of problems

\begin{equation}\label{eqn:general_minimization}
	\begin{split}
		\underset{\mathbf{a},\mathbf{u}}{\min} ~&~ p[\mathbf{u}]\\
		\text{ s.t } & \left\{\begin{array}{l}
			\mathbf{G(a)}\succeq 0\\
			\mathbf{L(a)u=f}
		\end{array}\right.
	\end{split}
\end{equation}
that build a unifying framework for compliance and peak power minimizations, namely, the minimization of a semi-definite representable (SDr) function under the equilibrium condition. By the equilibirum condition, we understand that the state variable $\mathbf{u}$ and design variable $\mathbf{a}$ are involved in the optimization problem and are constrained by the equilibrium condition $\mathbf{L(a)u=f}$.

\section{Minimization of SDr function under equilibrium}\label{sec:minSDR}
\MTc{There should be an intro paragraph stating what is inside this section}

Let us develop the framework of minimization of a semidefinite representable (SDr) function under the equilibrium
condition in more detail.

\subsection{Schur's complement Lemma}
Schur's complement of a block matrix is an essential tool to transform the non-linear constraints of a special structure into LMI constraints. Let us consider a block matrix $\mathbf{M}=\begin{pmatrix}\mathbf{A} & \mathbf{B}^T \\ \mathbf{B} & \mathbf{C}\end{pmatrix}$. Schur's complement Lemma relates the positive (semi)definiteness of $\mathbf{M}$ to the positive (semi)definiteness of its blocks. Let us first assume that the block $\mathbf{C}$ is invertible.
\begin{lemma}[Schur's complement Lemma with $\mathbf{C}$ invertible \citep{wolkowicz_handbook_2000}]
	Consider the block matrix $\mathbf{M}$ as defined above with the matrix $\mathbf{C}$ invertible. Then, $\mathbf{M}\succeq 0$ if and only if the following two conditions hold:
	\begin{equation}
		\left\{\begin{array}{l}
		\mathbf{C}\succeq 0,\\
		\mathbf{A} - \mathbf{B}^T\mathbf{C}^{-1}\mathbf{B}\succeq 0.
		\end{array}\right.
	\end{equation}
\end{lemma}
In the case when $\mathbf{C}$ is not invertible, we can use the Moore-Penrose pseudo-inverse instead of the inverse. However, additional conditions for $\mathbf{B}$ are now required, see \citep[Appendix A.5]{boyd_convex_2004}.
\begin{lemma}[Generalized Schur's complement Lemma \citep{boyd_convex_2004}]
	Consider a block matrix $\mathbf{M}$. $\mathbf{M}\succeq 0$ iff the following conditions hold:
	\begin{equation}
		\left\{\begin{array}{l}
		\mathbf{C}\succeq 0,\\
		(\mathbf{I}-\mathbf{C}\mathbf{C}^\dagger)\mathbf{B}=0,\\
		\mathbf{A} - \mathbf{B}^T\mathbf{C}^{\dagger}\mathbf{B}\succeq 0.
		\end{array}\right.
	\end{equation}
\end{lemma}
The second condition $(\mathbf{I}-\mathbf{CC}^\dagger)\mathbf{B}=0$ means that all the column vectors of $B$ are in the range space of $\mathbf{C}$. The use of Schur's complement Lemma is essential for obtaining equivalent (convex) reformulations of problems under the equilibrium equation $\mathbf{L(a)u=f}$, see \citep{tyburec_global_2021}, for example. The Generalized Schur's complement Lemma also allows us to conclude that $\mathbf{B}$'s columns are in the range space of $\mathbf{C}$, once there exists indeed a matrix $\mathbf{A}$ of an appropriate dimension such that
\begin{equation}
    \begin{pmatrix}\mathbf{A} & \mathbf{B}^T \\ \mathbf{B} & \mathbf{C}\end{pmatrix}\succeq 0.
\end{equation}

The (Generalized) Schur complement Lemma still holds whenever $\mathbf{A}$ and $\mathbf{C}$ are Hermitian and $\mathbf{B}^T$ is replaced by $\mathbf{B}^*$, this property is used to construct the constraints for peak power minimization since we use complex Fourier coefficients.

\subsection{Semidefinite representable function}
In \eqref{eqn:general_minimization}, we deal with the case where $p$ is a semidefinite representable function.
\begin{definition}[SDr function, {\citep[Lecture 4.2]{ben-tal_lectures_2001}}]\label{def:sdprep_1}
Let $p$ be a convex function; it is semidefinite representable (SDr) if and only if its epigraph is an LMI shadow. Namely, there are linear matrix-valued functions $P_j$, $j \in \{0,1,2\}$, such that
	\begin{equation}
        \begin{split}
                &\forall \theta \in \mathbb{R},\forall \mathbf{u}\in\mathbb{R}^n,\theta\geq p[\mathbf{u}]\\
            \iff& \exists \mathbf{w}\in\mathbb{R}^m, \mathbf{P}_0(\theta)+\mathbf{P}_1(\mathbf{u})+\mathbf{P}_2(\mathbf{w})\succeq 0.
        \end{split}
	\end{equation}
\end{definition}
Alternatively, a second type of semidefinite representation can be used. A function whose epigraph admits a semidefinite representation of the second type is also called the SDr function in this work.

\begin{definition}[SDr function, second representation]\label{def:sdprep_2}
Let $p$ be a convex function; it is semidefinite representable (SDr) if and only if there exists $\mathbf{b}\in\mathbb{R}^m$, together with $\mathbf{g}_k\in\mathbb{R}^n$ and $\mathbf{A}_k\in \mathbb{S}^N$ for $k=1,\dots,m$, such that :
\begin{equation}\label{eqn:sdr_second_type}
    \begin{split}
        &\forall \theta\in\mathbb{R}, \forall \mathbf{u}\in\mathbb{R}^n,\theta\geq p[\mathbf{u}]\\
        \iff& \exists \mathbf{Q}\in\mathbb{S}^N_+, b_k\theta+\mathbf{g}_k^T\mathbf{u} = \tr\{\mathbf{A}_k^T\mathbf{Q}\}.
    \end{split}
\end{equation}
\end{definition}

Whenever $p$ is SDr, the minimization problem \eqref{eqn:general_minimization} is, after adding the variables $\theta$ and $\mathbf{w}$, equivalent to
\begin{equation}\label{eqn:SDr_minimization}
	\begin{split}
		\underset{\mathbf{a},\mathbf{u},\theta,\mathbf{w}}{\min} ~&~ \theta\\
		\text{ s.t } & \left\{\begin{array}{l}
			a_i\geq0,m-\mathbf{q}^T\mathbf{a}\geq 0, \mathbf{L(a)}\succeq 0,\\
			\mathbf{L(a)u=f},\\
			\mathbf{P}_0(\theta)+\mathbf{P}_1(\mathbf{u})+\mathbf{P}_2(\mathbf{w})\succeq 0.
		\end{array}\right.
	\end{split}
\end{equation}
If the semidefinite representation of the second type is used, then we obtain the following equivalent optimization problem:
\begin{equation}\label{eqn:SDr_minimization_sec}
	\begin{split}
		\underset{\mathbf{a,u},\theta,\mathbf{Q}}{\min} ~&~ \theta\\
		\text{ s.t } & \left\{\begin{array}{l}
			a_i\geq0,m-\mathbf{q}^T\mathbf{a}\geq 0, \mathbf{L(a)}\succeq 0,\\
			\mathbf{L(a)u=f},\\
			b_k\theta + \mathbf{g}_k^T\mathbf{u} = \tr\{\mathbf{A}_k^T\mathbf{Q}\}, \forall k\in\{1,\dots,m\},\\
            \mathbf{Q}\succeq 0.
		\end{array}\right.
	\end{split}
\end{equation}
\begin{example}[Compliance minimization]\label{example:compl_min_1}
	Compliance minimization is formulated as
	\begin{equation}\tag{$\mathcal{P}_{\text{compl}}$}\label{eqn:compl_minimization}
		\begin{split}
			\underset{\mathbf{a},\mathbf{u}}{\min} ~&~ \mathbf{f}^T\mathbf{u}\\
			\text{ s.t } & \left\{\begin{array}{l}
				a_i\geq 0,\mathbf{q}^T\mathbf{a}\leq m, \mathbf{K(a)}\succeq 0,\\
				\mathbf{K(a)u=f}.
			\end{array}\right.
		\end{split}
	\end{equation}
	The compliance $\mathbf{u}\mapsto \mathbf{f}^T\mathbf{u}$ is SDr, as it is linear in $u$. $\theta\geq \mathbf{f}^T\mathbf{u}$ is already an LMI, with a one-dimensional matrix. Thus \eqref{eqn:compl_minimization} is equivalent to
    \begin{equation}
		\begin{split}
			\underset{\mathbf{a,u}, \theta}{\min} ~&~ \theta\\
			\text{ s.t } & \left\{\begin{array}{l}
				a_i\geq 0,\mathbf{q}^T\mathbf{a}\leq m, \mathbf{K(a)}\succeq 0,\\
				\mathbf{K(a)u=f},\\
                \theta\geq \mathbf{f}^T \mathbf{u}.
			\end{array}\right.
		\end{split}
	\end{equation}
\end{example}
As we can see in the example \ref{example:compl_min_1}, the semidefinite representation of compliance is trivial because it is reduced to a single inequality which is in fact an LMI with $1\times 1$ matrix. However, we should note that this does not change the difficulty of the compliance minimization because the non-convex constraint of equilibrium equation persists in the formulation. The compliance minimization has a very particular structure that allows us to reformulate it equivalently as a convex problem. This is classical in the literature, see for e.g. \citep{BentalBendsoe1993}, \citep{Achtziger1992} and \citep[lecture 4]{Ben_Tal_2001}. To explain the motivation of the method proposed in the current study, we reproduce the same convex reformulation of compliance minimization in the example \ref{example:compl_min2}.

\begin{example}[Peak power minimization under single harmonic load \citep{heidari_optimization_2009}]
	Assume for now that the time-varying load has one frequency component $\mathbf{f}(t)=\cos(\omega t)\mathbf{f}_R+\sin(\omega t)\mathbf{f}_I$. The nodal velocities $\mathbf{v}(t)=\cos(\omega t)\mathbf{v}_R+\sin(\omega t)\mathbf{v}_I$ satisfy the ODE
	\begin{equation}
		\mathbf{M(a)}\ddot{\mathbf{v}}+\mathbf{K(a)v}=\dot{\mathbf{f}}
	\end{equation}
	and at the steady state $\mathbf{v}_R$ and $\mathbf{v}_I$ satisfy
	\begin{equation}
		\left\{\begin{array}{l}
			\mathbf{K}_\omega(\mathbf{a})\mathbf{v}_R=\omega \mathbf{f}_I,\\
			\mathbf{K}_\omega(\mathbf{a})\mathbf{v}_I=-\omega \mathbf{f}_R.\\
		\end{array}\right.
	\end{equation}
	It has been shown in \citep{heidari_optimization_2009} that
	\begin{equation}
        \begin{split}
            \mathbf{f}(t)^T\mathbf{v}(t)&=\frac{e^{i2\omega t}}{4}\left[(\mathbf{f}_R^T\mathbf{v}_R-\mathbf{f}_I^T\mathbf{v}_I)-i(\mathbf{f}_R^T\mathbf{v}_I+\mathbf{f}_I^T\mathbf{v}_R)\right] \\
                &+ \frac{e^{-i2\omega t}}{4}\left[(\mathbf{f}_R^T\mathbf{v}_R-\mathbf{f}_I^T\mathbf{v}_I)+i(\mathbf{f}_R^T\mathbf{v}_I+\mathbf{f}_I^T\mathbf{v}_R)\right].
        \end{split}
	\end{equation}    
    Consider $(\theta,\mathbf{v})$ in the epigraph of the peak power function $\theta\geq \underset{t}{\max}{|\mathbf{f}(t)^T\mathbf{v}(t)|}$. Then, letting $z=e^{i\omega t}$, we can write $\mathbf{f}^T\mathbf{v}$ as a trigonometric polynomial in $z$. The epigraph condition is equivalent to the nonnegativity of polynomials $\theta\pm \mathbf{f}^T\mathbf{v}$. By Theorem \ref{thrm:pos_sdp},  iff $\exists \mathbf{Q}^0\in\mathbb{S}^3_+, \mathbf{Q}^1\in\mathbb{S}^3_+$:
	\begin{equation}\label{eqn:single_harmonic_representation1}
		\begin{split}
		    \theta = \tr\{\mathbf{\Lambda}_0\mathbf{Q}^0\},\\
            0 = \tr\{\mathbf{\Lambda}_{-1}\mathbf{Q}^0\},\\
            \frac{1}{4}((\mathbf{f}_R^T\mathbf{v}_R-\mathbf{f}_I^T\mathbf{v}_I)-i(\mathbf{f}_R^T\mathbf{v}_I+\mathbf{f}_I^T\mathbf{v}_R)) = \tr\{\mathbf{\Lambda}_{-2}\mathbf{Q}^0\}
		\end{split}
	\end{equation}
 and
	\begin{equation}\label{eqn:single_harmonic_representation2}
		\begin{split}
		    -\theta = \tr\{\mathbf{\Lambda}_0\mathbf{Q}^1\},\\
            0 = \tr\{\mathbf{\Lambda}_{-1}\mathbf{Q}^1\},\\
            -\frac{1}{4}((\mathbf{f}_R^T\mathbf{v}_R-\mathbf{f}_I^T\mathbf{v}_I)+i(\mathbf{f}_R^T\mathbf{v}_I+\mathbf{f}_I^T\mathbf{v}_R)) = \tr\{\mathbf{\Lambda}_{-2}\mathbf{Q}^1\},
		\end{split}
	\end{equation}
 where $\mathbf{\Lambda}_0$ is the $3$ by $3$ identity matrix, $\mathbf{\Lambda}_{-1}=\begin{pmatrix}
     0&0&0\\
     1&0&0\\
     0&1&0
 \end{pmatrix}$ and $\mathbf{\Lambda}_{-2}=\begin{pmatrix}
     0&0&0\\
     0&0&0\\
     1&0&0
 \end{pmatrix}$. 
 
Now, we perform several equivalent transformations of the optimization problem \eqref{eqn:peak_power_min} when the periodic load has one Fourier component. First, we replace the objective function $p[\mathbf{v_R},\mathbf{v_I}]=\max_{t}|\mathbf{f}(t)^T\mathbf{v}(t)|$, which is the peak power by its epigraph representation. This is done by introducing a slack variable $\theta$. We obtain:
\begin{equation}
    \begin{split}
        \underset{\mathbf{a},\mathbf{v_R},\mathbf{v_I},\theta}{\min}&\theta\\
        \text{s.t. }&\begin{cases}
            a_i\geq 0,\mathbf{q}^T\mathbf{a}\leq m, \mathbf{K}_\omega(\mathbf{a})\succeq 0,\\
            \mathbf{K}_\omega(\mathbf{a})\mathbf{v}_R=\omega \mathbf{f}_I,\\
			\mathbf{K}_\omega(\mathbf{a})\mathbf{v}_I=-\omega \mathbf{f}_R,\\
            \theta\geq p[\mathbf{v_R},\mathbf{v_I}]
        \end{cases}
    \end{split}
\end{equation}
Further, we use the fact that the peak power is SDP representable. The condition $\theta\geq p[\mathbf{v_R},\mathbf{v_I}]$ is equivalent to the existence of $\mathbf{Q}^{0}\succeq 0,\mathbf{Q}^{1}\succeq 0$ so that \eqref{eqn:single_harmonic_representation1} and \eqref{eqn:single_harmonic_representation2} hold with constant matrices $\mathbf{\Lambda}_j$ for $j=0,-1,-2$. Let us also introduce $\mathbf{Q}^0$ and $\mathbf{Q}^1$ as slack variables and use \eqref{eqn:single_harmonic_representation1} and \eqref{eqn:single_harmonic_representation2} as linear constraints, the minimization of the peak power under a single harmonic load reads as
\begin{equation}
	\begin{split}
		\underset{\mathbf{a},\mathbf{v}_R,\mathbf{v}_I,\theta,\mathbf{Q}^0,\mathbf{Q}^1}{\min} & \theta\\
		\text{ s.t. } & \left\{\begin{array}{l}
			a_i\geq 0, \mathbf{q}^T\mathbf{a}\leq m,\\
            \mathbf{K}_\omega(\mathbf{a})\succeq 0,\\
			\mathbf{K}_\omega(\mathbf{a})\mathbf{v}_R=\omega \mathbf{f}_I,\\
			\mathbf{K}_\omega(\mathbf{a})\mathbf{v}_I=-\omega \mathbf{f}_R,\\
			\text{\eqref{eqn:single_harmonic_representation1} and \eqref{eqn:single_harmonic_representation2} as linear constraints,}\\
            \mathbf{Q}^0,\mathbf{Q}^1\succeq 0.
		\end{array}\right.
	\end{split}
\end{equation}
\end{example}

\begin{example}[Peak power minimization]
	The peak power minimization problem is formalized as
	\begin{equation}
		\begin{split}
			\underset{\mathbf{a},\mathbf{c(v)}}{\min} ~&~ p[\mathbf{c(v)}]\\
			\text{ s.t } & \left\{\begin{array}{l}
				a_i\geq0, \mathbf{a}^T\mathbf{q}\leq m\\
                \mathbf{K}_{N\omega}(\mathbf{a}),\\
				\mathbf{K}_{k\omega}(\mathbf{a})\mathbf{c}_k(\mathbf{v})=ik\omega \mathbf{c}_k(\mathbf{f}),\forall k.
			\end{array}\right.
		\end{split}
	\end{equation}
	The peak power $p[c(v)]=\underset{t\in[0,\frac{2\pi}{\omega }]}{\max}\{\left\lvert f(t)^Tv(t)\right\lvert\}$ is SDr thanks to the semidefinite certificate of nonnegativity of trigonometric polynomials, which follows from Theorem \ref{thrm:pos_sdp} in Appendix \ref{appendix:A}. The peak power minimization is, therefore, equivalent to
	\begin{equation}\tag{$\mathcal{P}_{\text{pp}}$}\label{eqn:peak_power_min_SDr}
		\begin{split}
			\underset{\mathbf{a},\mathbf{c(v)},\theta,\mathbf{Q}^i}{\min} ~&~ \theta\\
			\text{ s.t } & \left\{\begin{array}{l}
				a_i\geq0, \mathbf{a}^T\mathbf{q}\leq m\\
                \mathbf{K}_{N\omega}(\mathbf{a})\succeq 0,\\
				\mathbf{K}_{k\omega}(\mathbf{a})\mathbf{c}_k(\mathbf{v})=ik\omega \mathbf{c}_k(\mathbf{f})\forall k\\
				\text{ the equality constraints \eqref{eqn:sdr_peak_power}},\\
                \mathbf{Q}^0,\mathbf{Q}^1\succeq 0.
			\end{array}\right.
		\end{split}
	\end{equation}
\end{example}
The difficulty of the minimization problem of the form \eqref{eqn:general_minimization} is now concentrated in the equilibrium equation constraint. All other constraints can be treated efficiently using existing SDP solvers. 

For the next developments, we define a so-called ``physical'' feasible point of the equilibrium equation and focus on the first type of semidefinite representation, Definition \ref{def:sdprep_1}. For the semidefinite representation of the second type, Definition \ref{def:sdprep_2}, we can easily adapt the results that will be presented now, and we will treat them in the end of this subsection.

We start by noting that if a particular pair of variables $\mathbf{(a,u)}$ is feasible for the constraint $\mathbf{L(a)u=f}$, then $\mathbf{f}$ is in the range of $\mathbf{L(a)}$, and $\mathbf{u}$ must be of the form $\mathbf{u=L(a)}^\dagger \mathbf{f}+\mathbf{u}_0$, where $\mathbf{L(a)}^\dagger$ is the Moore-Penrose pseudoinverse of $\mathbf{L(a)}$ and $\mathbf{u}_0$ is in the null space of $\mathbf{L(a)}$. We call $\mathbf{L(a)}^\dagger \mathbf{f}$ (resp. $\mathbf{u}_0$) the physical (resp. nonphysical) part of $\mathbf{u}$. Note that since $\mathbf{f}$ is in the range of $\mathbf{L(a)}$, $(\mathbf{a}, \mathbf{L(a)}^\dagger \mathbf{f})$ is feasible for the equilibrium condition. Thus, we make the following assumption.
\begin{assumption}\label{assumption:independance1}
	Assume that for any feasible $\mathbf{(a,u)}$ of the constraint $\mathbf{L(a)u=f}$, $p[\mathbf{u}]$ is independent of the nonphysical part of $\mathbf{u}$, i.e., $p[\mathbf{u}]=p[\mathbf{L(a)}^\dagger \mathbf{f}]$.
\end{assumption}
\begin{example}[Compliance minimization, continued]\label{example:compl_min2}
    For compliance minimization, this assumption is satisfied by the symmetry of $\mathbf{K(a)}$. If $\mathbf{f}$ is in the range of $\mathbf{K(a)}$ then $\mathbf{f}$ is orthogonal to $\mathbf{u_0}$, and thus, $p[\mathbf{u}]=\mathbf{f}^T\mathbf{u}=\mathbf{f}^T\mathbf{K(a)}^\dagger \mathbf{f}$. We can reformulate the compliance minimization as a linear SDP problem
	\begin{equation}
		\begin{split}
			\underset{\mathbf{a},\theta}{\min} ~&~ \theta\\
			\text{ s.t } & \left\{\begin{array}{l}
				a_i\geq 0,\mathbf{q}^T\mathbf{a}\leq m,\\
				\begin{pmatrix}
					\theta & \mathbf{f}^T \\
					\mathbf{f} & \mathbf{K(a)}
				\end{pmatrix}\succeq 0
			\end{array}\right.
		\end{split}
	\end{equation}
	thanks to generalized Schur's complement Lemma,
	\begin{equation}
		\left\{\begin{array}{l}
			\mathbf{K(a)}\succeq 0,\\
			\mathbf{f}\in\range{\mathbf{K(a)}},\\
			\theta\geq \mathbf{f}^T\mathbf{K(a)}^\dagger \mathbf{f},
		\end{array}\right.\iff \begin{pmatrix}
			\theta & \mathbf{f}^T \\
			\mathbf{f} & \mathbf{K(a)}
		\end{pmatrix}\succeq 0.
	\end{equation}
\end{example}
Because $p$ is independent of the nonphysical solution, its epigraph is also independent of the nonphysical solution for any feasible $\mathbf{(a,u)}$. By the SDP representability of $p$, the LMI representation of the epigraph should also be independent of the nonphysical solution. We recall that from Definition \ref{def:sdprep_1} we have
\begin{equation}
	\theta\geq p[\mathbf{u}] \iff \exists \mathbf{w}, P_0(\theta)+P_1(\mathbf{u})+P_2(\mathbf{w})\succeq 0.
\end{equation} 
Then, we can express the linear matrix function $P_1(u)$ as
\begin{equation}
	P_1(u)=\sum_j (g_j^Tu) \tilde{P}_j
\end{equation}
with a family of vectors $\mathbf{g}_j$ and matrices $\tilde{\mathbf{P}}_j$ of an appropriate size. We assume that the following condition holds:

\begin{assumption}\label{assumption:independance2}
	Assume that $\mathbf{g}_j$ is in the range of $\mathbf{L(a)}$ whenever $\mathbf{f}$ is in the range of $\mathbf{L(a)}$.
\end{assumption}
\begin{lemma}
    Assumption \ref{assumption:independance2} is sufficient for Assumption~\ref{assumption:independance1}.
\end{lemma}
\begin{proof}
    Let $p$ be an SDr function such that Assumption \ref{assumption:independance2} holds. Then for any $\mathbf{(a,u)}$ feasible for $\mathbf{L(a)u=f}$, by the symmetry of $\mathbf{L(a)}$, if $\mathbf{g}_j$ is in the range of $\mathbf{L(a)}$ then $\mathbf{u}_0$ is orthogonal to $\mathbf{g}_j$ and thus
    \begin{equation}
    	\mathbf{P}_1(\mathbf{u})=\sum_j (\mathbf{g}_j^T\mathbf{L(a)}^\dagger \mathbf{f}) \tilde{\mathbf{P}}_j.
    \end{equation}
    Thus, the LMI representation of $p$ is independent of the non-physical solution, which implies that $p$ is independent of the nonphysical solution.
\end{proof}

Returning to the peak power minimization problem, by a careful study of the generalized eigenvalue problem of free vibrations \citep{achtziger_structural_2008},  we can show that the peak power function is indeed independent of the nonphysical part of $(\mathbf{a},\mathbf{c(v)})$ whenever it is feasible and whenever the highest driving frequency of the load is below the smallest non-singular free-vibration eigenfrequency of the structure, i.e., $-N^2\omega^2\mathbf{M(a)}+\mathbf{K(a)}\succeq 0$. For the proof, we refer the reader to Lemma \ref{lemma:same_range} of Appendix \ref{appendix:B}.

Assumptions \ref{assumption:independance1} and \ref{assumption:independance2} are not restrictive, since it would be questionable to optimize an objective function that depends on the nonphysical part of the solution of $\mathbf{L(a)u=f}$. In the context of topology optimization of discrete structures, the non-physical part of $\mathbf{L(a)u=f}$ would consist of the displacements or the velocities of nodes without any attached member.

To check Assumption \ref{assumption:independance2} for objective function with the second SDP representation, we recall that if the objective function $p[\mathbf{u}]$ has the second type SDP representation then there exists $b_k\in\mathbb{R}$, $\mathbf{g}_k$ and $\mathbf{A}_k$ of fixed dimension for $k=1,\dots,m$ such that $\theta\geq p[\mathbf{u}]\iff b_k\theta + \mathbf{g}_k^T\mathbf{u} = \tr\{\mathbf{A}_k^TQ\}$ where $\mathbf{Q}\succeq 0$. $p$ satisfies the assumption \ref{assumption:independance2} iff the $\mathbf{g}_k$'s of the SDP representation are in the range of $\mathbf{L(a)}$ whenever $\mathbf{f}$ is in the range of $\mathbf{L(a)}$.

\subsection{Convex relaxation}\label{sec:conv_relax}
Let us now consider a SDr function and a minimization problem of the form \eqref{eqn:general_minimization}. Furthermore, let Assumptions \ref{assumption:independance1} and \ref{assumption:independance2} hold. In what follows, we propose a convex relaxation of \eqref{eqn:general_minimization}. For brevity, we only present the strategy for the first-type representation, Definition \ref{def:sdprep_1}. Adaptation of the results of this subsection to Definition \ref{def:sdprep_2} is straightforward.

First, by the semidefinite representability of the objective function, we introduce slack variables $\theta$ and $w$ and rewrite \eqref{eqn:general_minimization} as
\begin{equation}
	\begin{split}
		\underset{\mathbf{a,u},\theta,\mathbf{w}}{\min} ~&~ \theta\\
		\text{ s.t } & \left\{\begin{array}{l}
			a_i\geq0,m-\mathbf{q}^T\mathbf{a}\geq 0, \mathbf{L(a)}\succeq 0,\\
			\mathbf{L(a)u=f},\\
			\mathbf{P}_0(\theta)+\mathbf{P}_1(\mathbf{u})+\mathbf{P}_2(\mathbf{w})\succeq 0.\\
		\end{array}\right.
	\end{split}
\end{equation}
Due to Assumption \ref{assumption:independance1}, we can remove $\mathbf{u}$ from the set of optimization variables, since only the physical part of the equation $\mathbf{L(a)u=f}$ matters. However, we need to keep in mind that $\mathbf{f}$ must be in the range of $\mathbf{L(a)}$, 
\begin{equation}
	\begin{split}
		\underset{\mathbf{a},\theta,\mathbf{w}}{\min} ~&~ \theta\\
		\text{ s.t } & \left\{\begin{array}{l}
			a_i\geq0,m-\mathbf{q}^T\mathbf{a}\geq 0, \mathbf{L(a)}\succeq 0,\\
			\mathbf{f}\in\range{\mathbf{L(a)}},\\
			\mathbf{P}_0(\theta)+\mathbf{P}_1(\mathbf{L(a)}^\dagger \mathbf{f})+\mathbf{P}_2(\mathbf{w})\succeq 0.\\
		\end{array}\right.
	\end{split}
\end{equation}
After exploiting Assumption \ref{assumption:independance2}, we observe that the linear matrix function $\mathbf{P}_2$ evaluated at $\mathbf{L(a)}^\dagger \mathbf{f}$ is
\begin{equation}
	\mathbf{P}_1(\mathbf{L(a)}^\dagger \mathbf{f}) = \sum_{j}(\mathbf{g}_j^T\mathbf{L(a)}^\dagger \mathbf{f}) \tilde{\mathbf{P}}_j.
\end{equation} 
Furthermore, we consider a matrix $\mathbf{F} = (\mathbf{f}~\mathbf{g}_j)$ whose column vectors are $\mathbf{f}$ and $\mathbf{g}_j$ for all (finitely many) $j$. Then, there exist constant matrices $\mathbf{C}_j$ such that $\mathbf{g}_j^T\mathbf{L(a)}^\dagger \mathbf{f} = \tr\{\mathbf{C}_j^T(\mathbf{F}^T\mathbf{L(a)}^\dagger \mathbf{F})\}$. To derive $\mathbf{C}_j$, we note that
\begin{equation}
    \begin{split}
        \mathbf{g}_j^T\mathbf{L(a)}^\dagger \mathbf{f} &= (\mathbf{Fe}_{j+1})^T\mathbf{L(a)}^\dagger (\mathbf{Fe}_0) \\
        &= \tr\left\{\mathbf{e}_0^T\mathbf{e}_{j+1} \mathbf{F}^T\mathbf{L(a)}^\dagger \mathbf{F}\right\}
    \end{split}
\end{equation}
with $\mathbf{e}_i$ denoting vectors of the canonical basis. Consequently, we can symmetrize the last expression and obtain $\mathbf{C}_j$. After introducing a variable $\mathbf{X}$ and a constraint $\mathbf{X}=\mathbf{F}^T\mathbf{L(a)}^\dagger \mathbf{F}$, we obtain the optimization problem
\begin{equation}
	\begin{split}
		\underset{\mathbf{a},\theta,\mathbf{w},\mathbf{X}}{\min} ~&~ \theta\\
		\text{ s.t } & \left\{\begin{array}{l}
			a_i\geq0,m-\mathbf{q}^T\mathbf{a}\geq 0, \mathbf{L(a)}\succeq 0,\\
			\mathbf{f}\in\range{\mathbf{L(a)}},\\
			\mathbf{X}=\mathbf{F}^T\mathbf{L(a)}^\dagger \mathbf{F},\\
			\mathbf{P}_0(\theta)+\sum_j\tr\{\mathbf{C}_j^T\mathbf{X}\} \tilde{\mathbf{P}}_j+\mathbf{P}_3(\mathbf{w})\succeq 0.\\
		\end{array}\right.
	\end{split}
\end{equation}
The range constraint can now be replaced because $\mathbf{X}=\mathbf{F}^T\mathbf{L(a)}^\dagger \mathbf{F}$ is equivalent to
\begin{equation}
   \mathbf{X}=\mathbf{F}^T\mathbf{L(a)}^\dagger \mathbf{F}\iff\left\{\begin{array}{l}
         \mathbf{X}\preceq \mathbf{F}^T\mathbf{L(a)}^\dagger \mathbf{F}\\
         \mathbf{X}\succeq \mathbf{F}^T\mathbf{L(a)}^\dagger \mathbf{F}
    \end{array}\right.,
\end{equation}
by splitting matrix equality into two semidefinite inequalities. Combined with Assumption \ref{assumption:independance2} and thanks to Schur's complement Lemma, we have the equivalent formulation
\begin{equation}
	\begin{split}
		\underset{\mathbf{a},\theta,\mathbf{w},\mathbf{X}}{\min} ~&~ \theta\\
		\text{ s.t } & \left\{\begin{array}{l}
			a_i\geq0,m-\mathbf{q}^T\mathbf{a}\geq 0,\\
			\begin{pmatrix}
			    \mathbf{X} & \mathbf{F}^T\\
                \mathbf{F} & \mathbf{L(a)}
			\end{pmatrix}\succeq 0,\\
			\mathbf{X}\preceq \mathbf{F}^T\mathbf{L(a)}^\dagger \mathbf{F},\\
			\mathbf{P}_0(\theta)+\sum_j\tr\{\mathbf{C}_j^T \mathbf{X}\} \tilde{\mathbf{P}}_j+\mathbf{P}_3(\mathbf{w})\succeq 0.\\
		\end{array}\right.
	\end{split}
\end{equation}
For any $\mathbf{X}$ and $\mathbf{a}$ such that $\begin{pmatrix}\mathbf{X} & \mathbf{F}^T\\\mathbf{F} & \mathbf{L(a)}\end{pmatrix}\succeq 0$, there is the trace inequality $\tr\{\mathbf{X}-\mathbf{F}^T\mathbf{L(a)}^\dagger \mathbf{F}\}\geq 0$. The trace equality $\tr\{\mathbf{X}-\mathbf{F}^T\mathbf{L(a)}^\dagger \mathbf{F}\}= 0$ is thus sufficient and necessary for the satisfaction of $\mathbf{X}\preceq \mathbf{F}^T\mathbf{L(a)}^\dagger \mathbf{F}$. Consequently, we can replace $\mathbf{X}\preceq \mathbf{F}^T\mathbf{L(a)}^\dagger \mathbf{F}$ by $\tr\{\mathbf{X}-\mathbf{F}^T\mathbf{L(a)}^\dagger \mathbf{F}\}= 0$.

The only non-convex and non-linear constraint is now $\tr\{\mathbf{X}-\mathbf{F}^T\mathbf{L(a)}^\dagger \mathbf{F}\}= 0$.  A convex relaxation consists of removing the trace equality constraint,
\begin{equation}\label{eqn:convex_relaxation}
	\begin{split}
		\underset{\mathbf{a},\theta,\mathbf{w},\mathbf{X}}{\min} ~&~ \theta\\
		\text{ s.t } & \left\{\begin{array}{l}
			a_i\geq0,m-\mathbf{q}^T\mathbf{a}\geq 0,\\
			\begin{pmatrix}
			    \mathbf{X} & \mathbf{F}^T\\
                \mathbf{F} & \mathbf{L(a)}
			\end{pmatrix}\succeq 0,\\
			\mathbf{P}_0(\theta)+\sum_j\tr\{\mathbf{C}_j^T \mathbf{X}\} \tilde{\mathbf{P}}_j+\mathbf{P}_3(\mathbf{w})\succeq 0.\\
		\end{array}\right.
	\end{split}
\end{equation}

\begin{remark}
	We observe that:
	\begin{itemize}
		\item The relaxation does not require linearity of $\mathbf{L(a)}$, i.e., it applies to both truss and frame settings. The relaxation requires only that $\mathbf{L(a)} \succeq 0$. Under the assumption of linearity of $\mathbf{L(a)}$, the relaxation is indeed a convex SDP problem, and thus can be solved efficiently by existing software.
		\item The feasible points of the relaxation are statically admissible designs, i.e., cross-section areas that carry the structural loads, which holds because the second LMI constraint secures that $\mathbf{f}$ is in the range of $\mathbf{L(a)}$.
		\item If $\mathbf{X}=\mathbf{F}^T\mathbf{L(a)}^\dagger \mathbf{F}$ holds at the optimal solution of the convex relaxation, then we can extract the optimal solution of \eqref{eqn:general_minimization} from the convex relaxation.

	\end{itemize}
\end{remark}
\subsection{Link between the convex relaxation and Lagrange relaxation}\label{sec:lag_relax}
Let us now take another look at the equivalent formulation with trace equality:
\begin{equation}
	\begin{split}
		\underset{\mathbf{a},\theta,\mathbf{w},\mathbf{X}}{\min} ~&~ \theta\\
		\text{ s.t } & \left\{\begin{array}{l}
			a_i\geq0,m-\mathbf{q}^T\mathbf{a}\geq 0,\\
			\begin{pmatrix}
			    \mathbf{X} & \mathbf{F}^T\\
                \mathbf{F} & \mathbf{L(a)}
			\end{pmatrix}\succeq 0,\\
            \tr\{\mathbf{X}-\mathbf{F}^T\mathbf{L(a)}^\dagger \mathbf{F}\}= 0\\
			\mathbf{P}_0(\theta)+\sum_j\tr\{\mathbf{C}_j^T \mathbf{X}\} \tilde{\mathbf{P}}_j+\mathbf{P}_3(\mathbf{w})\succeq 0.\\
		\end{array}\right.
	\end{split}
\end{equation}
Considering the second LMI, if $\mathbf{a}$ and $\mathbf{X}$ are feasible, it must hold that
\begin{equation}
    \tr\{\mathbf{X}-\mathbf{F}^T\mathbf{L(a)}^\dagger \mathbf{F}\} \geq 0.
\end{equation}
Hence, we can replace the equality by $\tr\{\mathbf{X}-\mathbf{F}^T\mathbf{L(a)}^\dagger \mathbf{F}\} \leq 0$ and the following optimization stays equivalent to \eqref{eqn:general_minimization}:
\begin{equation}\label{eqn:before_lag_trace_ineq}
		\begin{split}
		\underset{\mathbf{a},\theta,\mathbf{w},\mathbf{X}}{\min} ~&~ \theta\\
		\text{ s.t } & \left\{\begin{array}{l}
			a_i\geq0,m-\mathbf{q}^T\mathbf{a}\geq 0,\\
			\begin{pmatrix}
			    \mathbf{X} & \mathbf{F}^T\\
                \mathbf{F} & \mathbf{L(a)}
			\end{pmatrix}\succeq 0,\\
            \tr\{\mathbf{X}-\mathbf{F}^T\mathbf{L(a)}^\dagger \mathbf{F}\} \leq 0\\
			\mathbf{P}_0(\theta)+\sum_j\tr\{\mathbf{C}_j^T \mathbf{X}\} \tilde{\mathbf{P}}_j+\mathbf{P}_3(\mathbf{w})\succeq 0.\\
		\end{array}\right.
	\end{split}
\end{equation}
The Lagrange relaxation consists of moving the inequality constraint into the objective and penalizing it with a positive weight $\eta$, leading to
\begin{equation}\label{eqn:lag_relax}
		\begin{split}
		\underset{\mathbf{a},\theta,\mathbf{w},\mathbf{X}}{\min} ~&~ \theta+\eta\tr\{\mathbf{X}-\mathbf{F}^T\mathbf{L(a)}^\dagger \mathbf{F}\}\\
		\text{ s.t } & \left\{\begin{array}{l}
			a_i\geq0,m-\mathbf{q}^T\mathbf{a}\geq 0,\\
			\begin{pmatrix}
			    \mathbf{X} & \mathbf{F}^T\\
                \mathbf{F} & \mathbf{L(a)}
			\end{pmatrix}\succeq 0,\\
			\mathbf{P}_0(\theta)+\sum_j\tr\{\mathbf{C}_j^T \mathbf{X}\} \tilde{\mathbf{P}}_j+\mathbf{P}_3(\mathbf{w})\succeq 0.\\
		\end{array}\right.
	\end{split}
\end{equation}
where the penalty factor $\eta$ is increased until the original constraint $\tr\{\mathbf{X}-\mathbf{F}^T\mathbf{L(a)}^\dagger \mathbf{F}\}\leq 0$ becomes feasible. Let us denote by $L(\eta)$ the optimal value depending on $\eta$. For one $\eta^*$, suppose that $(\mathbf{a}^*,\theta^*,\mathbf{w}^*,\mathbf{X}^*)$ is the corresponding solution, then for any $(\mathbf{a},\theta,\mathbf{w},\mathbf{X})$ satisfying the constraint of \eqref{eqn:before_lag_trace_ineq}, we have
\begin{equation}
    \begin{split}
        \theta&\geq\theta+\underbrace{\eta^*\tr\{\mathbf{X}-\mathbf{F}^T\mathbf{L(a)}^\dagger \mathbf{F}\}}_{\leq 0}\\
        &\geq \theta^*+\eta^*\tr\{\mathbf{X}^*-\mathbf{F}^T\mathbf{L}(\mathbf{a}^*)^\dagger \mathbf{F}\} = L(\eta^*).
    \end{split}
\end{equation}
Therefore, $\underset{\eta\geq 0}{\max}~L(\eta)$ provides a lower bound of the problem \eqref{eqn:before_lag_trace_ineq}.

We note that the convex relaxation \eqref{eqn:convex_relaxation} corresponds to the Lagrange relaxation with weight $\eta=0$. However, the Lagrange relaxation remains a difficult problem due to the presence of $\tr\{\mathbf{F}^T\mathbf{L(a)}^\dagger \mathbf{F}\}$. We propose to solve the following relaxation, by removing $\tr\{\mathbf{F}^T\mathbf{L(a)}^\dagger \mathbf{F}\}$ from the objective function:
\begin{equation}\label{eqn:lag_relax_drop}
	\begin{split}
		\underset{\mathbf{a},\theta,\mathbf{w},\mathbf{X}}{\min} ~&~ \theta+\eta\tr\{\mathbf{X}\}\\
		\text{ s.t } & \left\{\begin{array}{l}
			a_i\geq0,m-\mathbf{q}^T\mathbf{a}\geq 0,\\
			\begin{pmatrix}
			    \mathbf{X} & \mathbf{F}^T\\
                \mathbf{F} & \mathbf{L(a)}
			\end{pmatrix}\succeq 0,\\
			\mathbf{P}_0(\theta)+\sum_j\tr\{\mathbf{C}_j^T \mathbf{X}\} \tilde{\mathbf{P}}_j+\mathbf{P}_3(\mathbf{w})\succeq 0.\\
		\end{array}\right.
	\end{split}
\end{equation}
which is a linear SDP problem. 

\MTc{Returning to the comment by Didier on Thursday, $\eta\tr\{X-F^TL(a)^\dagger F\}$ is equivalent to $\eta\tr\{X\}$ - $\eta\tr\{F^TL(a)^\dagger F\}$. For the first term, we want as low a trace of $X$ as possible. For the second term and due to the negative sign, we would like to have the maximum trace of the pseudo-inverse (weighted by $F$), which can be roughly approximated by minimizing the trace of $L(a)$. Consequently, we may have $\eta\tr\{X\} + \eta\tr\{L(a)\}$ in the objective. I don't know whether there is any good in that, however, because the term $\tr\{L(a)\}$ has a similar effect as the volume/weight function $q^Ta$. Because of the penalty, the relaxed solution may then tend to have a very low weight (not desired; we expect the opposite for the optimum points). To conclude this reasoning, I think you can also justify the only present $X$ by this.}

\section{Numerical examples}\label{sec:examples}

This section illustrates the theoretical developments of Section \ref{sec:minSDR} by means of numerical experiments. The construction of the SDP problems is detailed in Appendix \ref{appendix:peak_power_minimization_multi}. Specifically, we solve a problem introduced in \citep{heidari_optimization_2009}, showing that our approach not only reaches the optimal solution for the in-phase setting, but also provides a better design for the out-of-phase scenario. Our second example \ref{sec:cantilever} illustrates the applicability of the method to larger problems. The third \ref{sec:two_rotation} and fourth \ref{sec:multi_freq} examples illustrate the influence of multiple harmonics on design. 

All numerical examples presented in this section were implemented in a Python code available at \url{https://gitlab.com/ma.shenyuan/truss_exp}, with the optimization programs modeled in \textbf{CVXPY} \citep{diamond2016cvxpy}, a Python interface for modeling optimization problems, and providing an interface to the optimizer \textbf{MOSEK} \citep{mosek} that we used to solve the resulting SDP problems. The optimization problems were solved on a standard laptop, equipped with the 6 AMD Ryzen 5 4500U processors and $16$ GB of RAM.

\subsection{$21$-element problem by Heidari~\emph{et al.}}\label{sec:examples_heidari}
As the first illustration, we consider the problem introduced in \citep{heidari_optimization_2009}: a ground structure containing $21$ finite elements and $12$ nodes, see Fig.~\ref{fig:heidari_bc}. In this problem, all horizontal and vertical bars share a length of $1$, whereas the diagonal ones are $\sqrt{2}$ long. All elements are made of the same material, with the Young modulus $E=25,000$ and density $\rho=1.0$. Furthermore, we bound the total structural weight from above by $m=1$.

\begin{figure}[!htbp]
\centering
\includegraphics{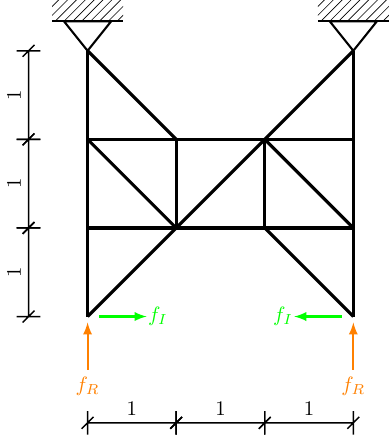}
\caption{Boundary conditions and ground structure of the $21$-element truss structure introduced in \citep{heidari_optimization_2009}.}\label{fig:heidari_bc}
\end{figure}

\begin{figure*}[!h]
\centering
\includegraphics[width=\linewidth]{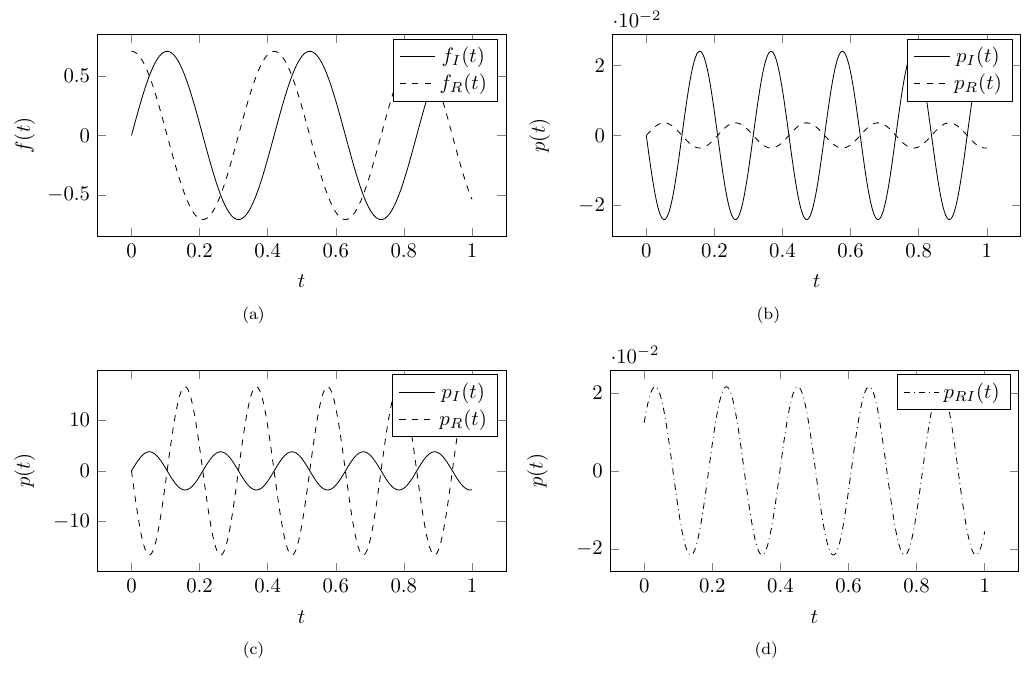}
\caption{(a) Time-varying loads $f_R(t)$ and $f_I(t)$ and the powers associated with the (b) optimal designs under in-phase loads, and (c) the same designs under out-of-phase loads. Figure (d) shows the power of a design optimized for out-of-phase loads.}
\label{fig:powers}
\end{figure*}

Kinematic boundary conditions consist of fixed supports at the top nodes, preventing their horizontal and vertical movements. For dynamic loads, we consider them acting at the bottom nodes, with the horizontal forces (denoted in green color in Fig.~\ref{fig:heidari_bc}) $f_I(t) = \frac{1}{2} \sin(\omega t)$ and the vertical forces (drawn in orange in Fig.~\ref{fig:heidari_bc}) being $f_R(t)=\frac{1}{2} \cos(\omega t)$. For both these loading functions, we assume the same angular frequency $\omega=15$. The time-dependent force magnitudes appear visualized in Fig.~\ref{fig:powers}a. Notice that the resultant of the components $f_R$ and $f_I$ is always a force of magnitude $1/2$, thus modeling an unbalanced rotating load with the period of $T=\frac{2\pi}{15}\approx0.42$. Under this setting, the time varying load has one Fourier component $\mathbf{c}_1(\mathbf{f})$ such that the last four values are
\begin{equation}
    \mathbf{c}_1(\mathbf{f})=\begin{pmatrix}
        \vdots\\
        i/4\\
        1/4\\
        -i/4\\
        1/4
    \end{pmatrix}.
\end{equation}

\subsubsection{In-phase loads}
First, we optimize the peak power of the loads $f_R(t)$ and $f_I(t)$ independently. For this setting, the relaxation with penalty \eqref{eqn:lag_relax_drop} is exact, as was shown in \citep{heidari_optimization_2009}. Thus, globally optimal solutions can be found by solving a single SDP problem; see Figs.~\ref{fig:in_phase} and \ref{fig:in_phase2} for their topology and Fig.~\ref{fig:powers}b for the optimal time-varying powers $p_R(t)$ and $p_I(t)$. The corresponding optimal peak powers are $0.0036$ and $0.0241$, respectively.

\begin{figure*}[!h]
\centering
\subfigure[]{
         \includegraphics[height=4cm]{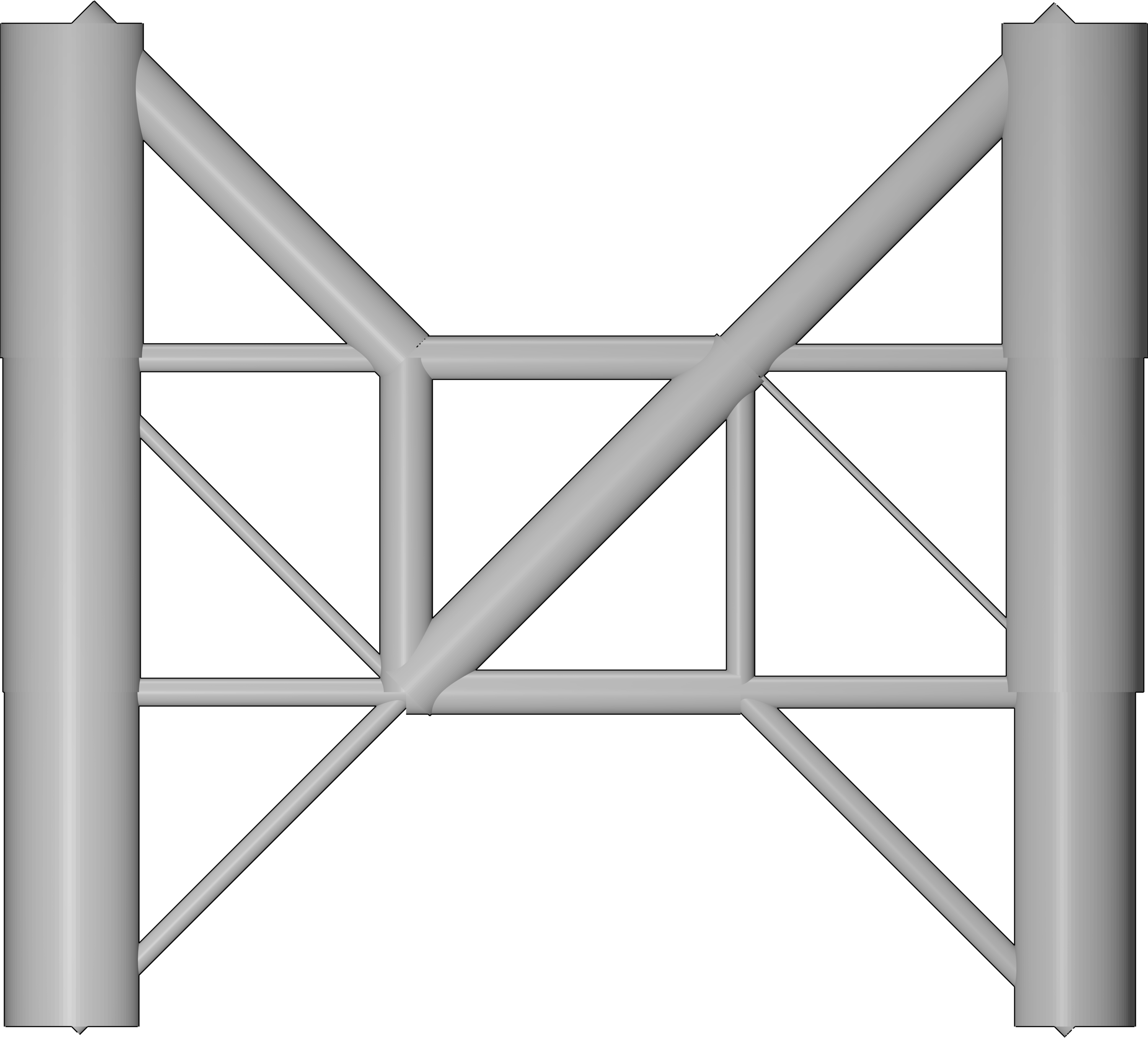}
        \label{fig:in_phase}
    }%
\hfill\subfigure[]{
         \includegraphics[height=4cm]{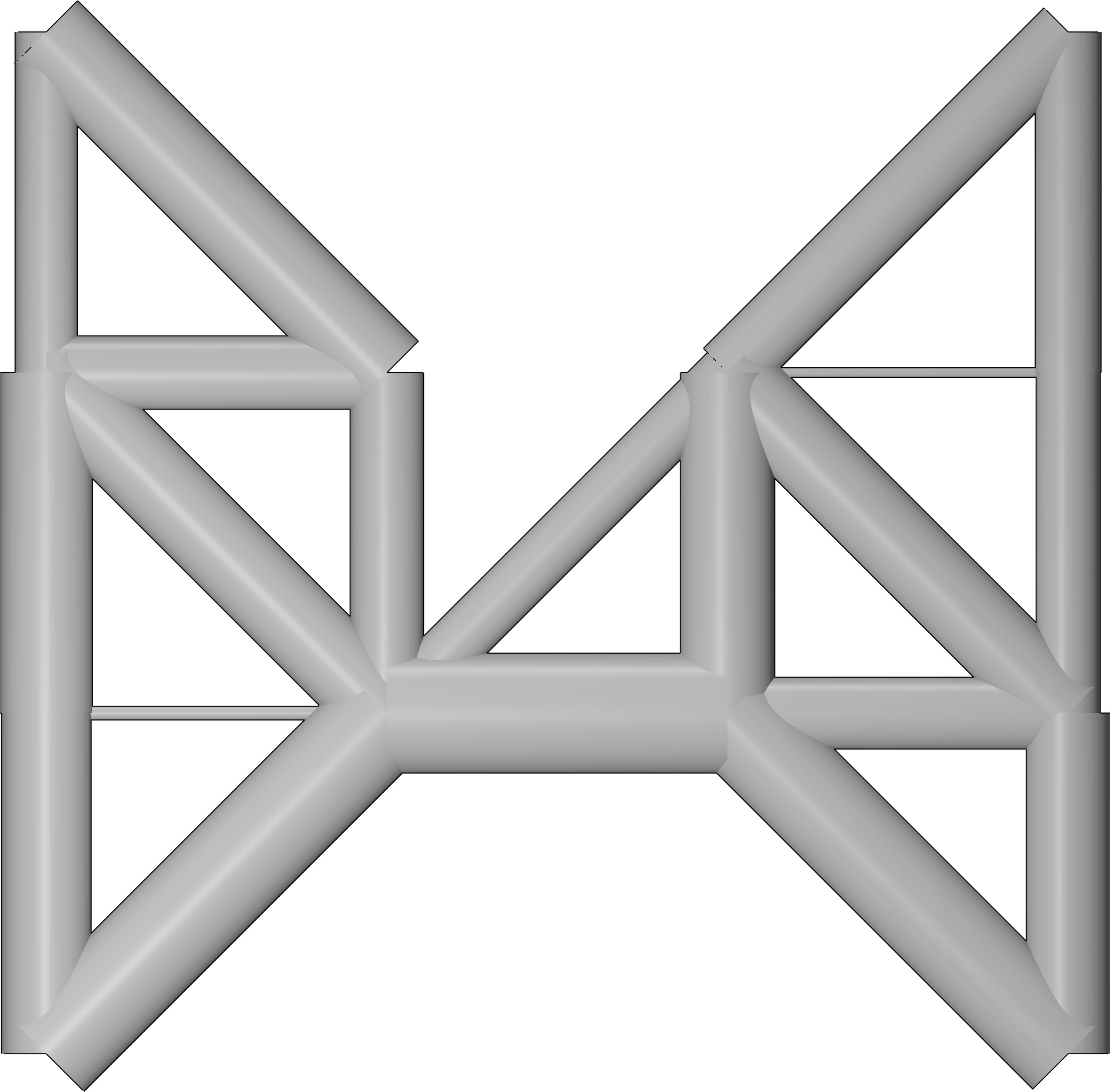}
        \label{fig:in_phase2}
    }%
\hfill\subfigure[]{
         \includegraphics[height=4cm]{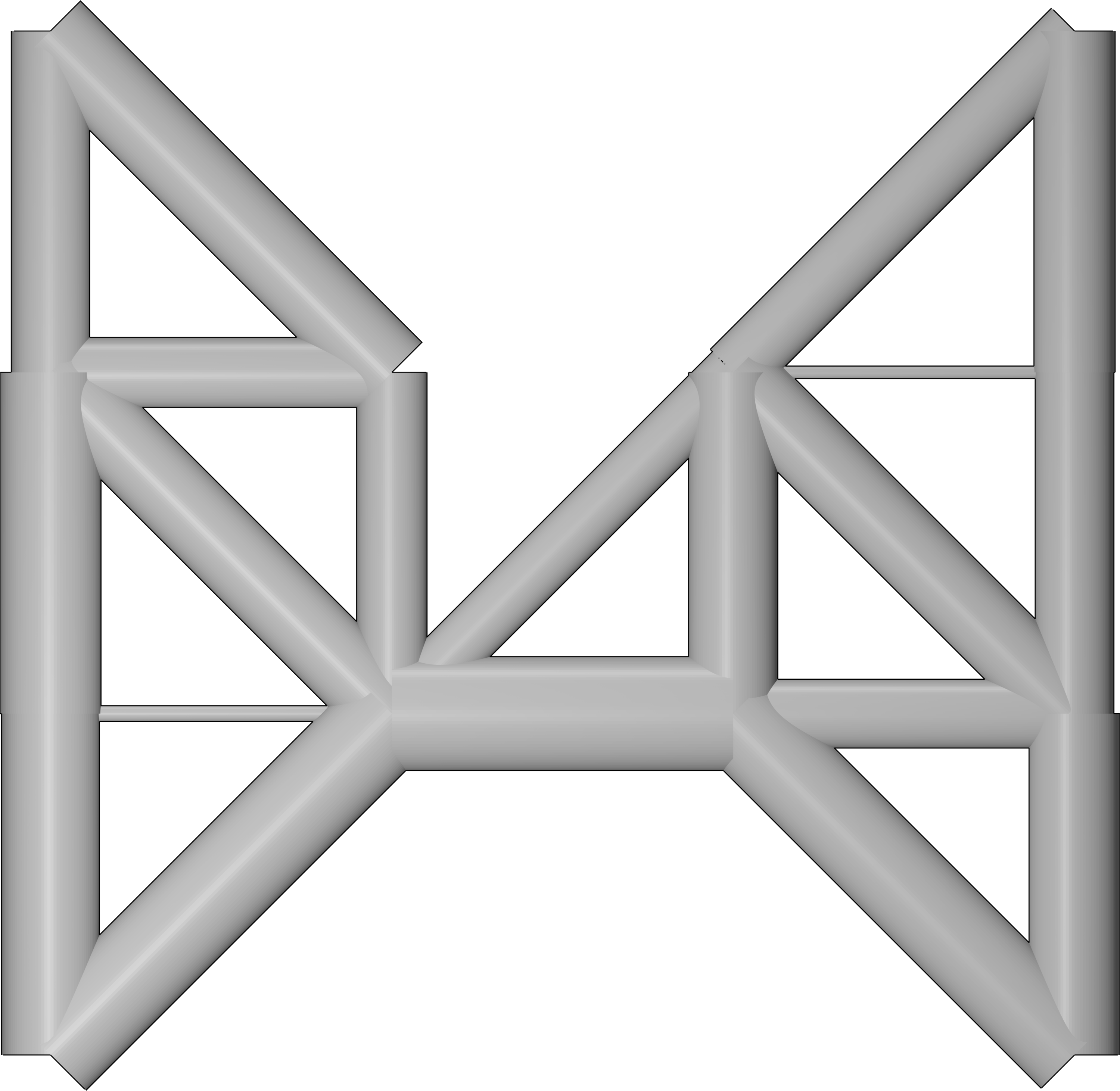}
        \label{fig:out_phase2}
    }
\caption{Optimized topology for the (a) in-phase loads $f_R$, (b) in-phase loads $f_I$, and (c) out-of-phase loads $f_R$ and $f_I$ acting according to Fig.~\ref{fig:heidari_bc}.}
\end{figure*}

\subsubsection{Out-of-phase loads}
Next, we consider the setting of the loads $f_I(t)$ and $f_R(t)$ that act simultaneously. The first, naive option, would be to investigate the performance of structures designed for in-phase loads $f_R(t)$ or $f_I(t)$ only. Not surprisingly, the performance of these designs is far from optimal because they do not suppress power by exploiting the interaction between the loads; see Fig.~\ref{fig:powers}c. A similar situation also appears for the worst-case peak power minimization by Heidari \textit{et al.} \citep{heidari_optimization_2009}, in which a solution is searched for the worst-case of all unit forces, i.e., the loads acting in phase. As shown in Fig.~\ref{fig:powers}d, the peak power for the out-of-phase configuration can be considerably reduced by enabling this interaction.

To achieve this, we apply the convex relaxation introduced in Section \ref{sec:minSDR}. The optimal value of the objective function is equal to $1.203\times 10^{-11}$, which is a lower bound for the minimal peak power. However, the actual peak power of the design given by the convex relaxation is equal to $0.0831$. For the relaxed solution, we also observe that the total mass constraint is not active: The total mass for the relaxation solution is only $0.40$, which means that we have used $40\%$ of the available mass.

Also, the optimal $\mathbf{X}$ is different from $\mathbf{F}^*\mathbf{L}_{1,\omega}(\mathbf{a})^\dagger \mathbf{F}$ at the optimal solution. The trace difference is evaluated as $\tr\{\mathbf{X}-\mathbf{F}^*\mathbf{L}_{1,\omega}(\mathbf{a})^\dagger \mathbf{F}\}=1508.038$, which is far from zero.

Furthermore, we implement the penalized relaxation \eqref{eqn:lag_relax_drop} with $\eta=10$, chosen large enough for the equality $\tr\{\mathbf{X}-\mathbf{F}^*\mathbf{L}_{1,\omega}(\mathbf{a})^\dagger \mathbf{F}\}=0$ to hold, and also render the mass bound active.

We numerically confirm that $\theta$ is indeed equal to the peak power $0.0216$, and there is a trace equality $\mathbf{X}=\mathbf{F}^*\mathbf{L}_{1,\omega}(\mathbf{a})^\dagger \mathbf{F}$. The optimized design $\mathbf{a}$ is shown in Figure \ref{fig:pen_single_harmonic}, with the line widths illustrating the actual value of $\mathbf{a}$.

\begin{figure}[H]
    \centering
    \includegraphics[width=0.5\textwidth]{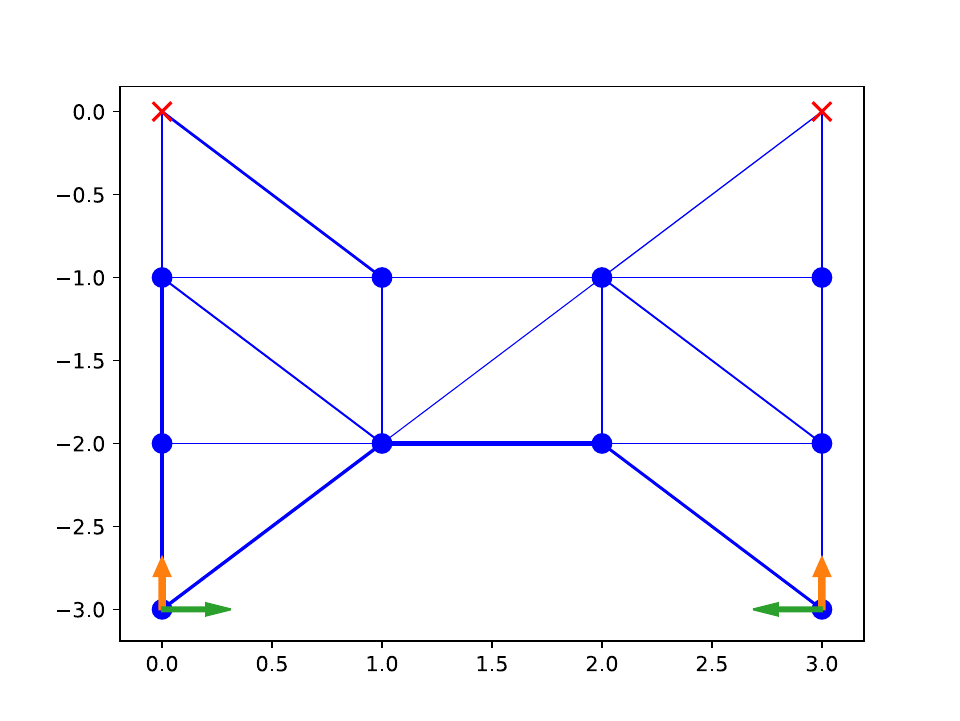}
    \caption{Optimal solutions of convex relaxation with penalty parameter $\eta=10$.}
    \label{fig:pen_single_harmonic}
\end{figure}

As we can see in the above optimization, a large enough $\eta$ results in trace equality. It remains to investigate the effect of $\eta$ on the peak power performance of the optimal solution. To this end, we solved the SDP relaxation for $80$ different values of $\eta$ ranging from $10^{-9}$ to $10$. We collected the trace differences, total masses and optimal $\theta$ of the solutions of the SDP relaxations. In addition, we also plot the actual peak power of the optimal design in an orange color, see Figure \ref{fig:many_eta}.

\begin{figure}[H]
    \centering
    \includegraphics[width=0.5\textwidth]{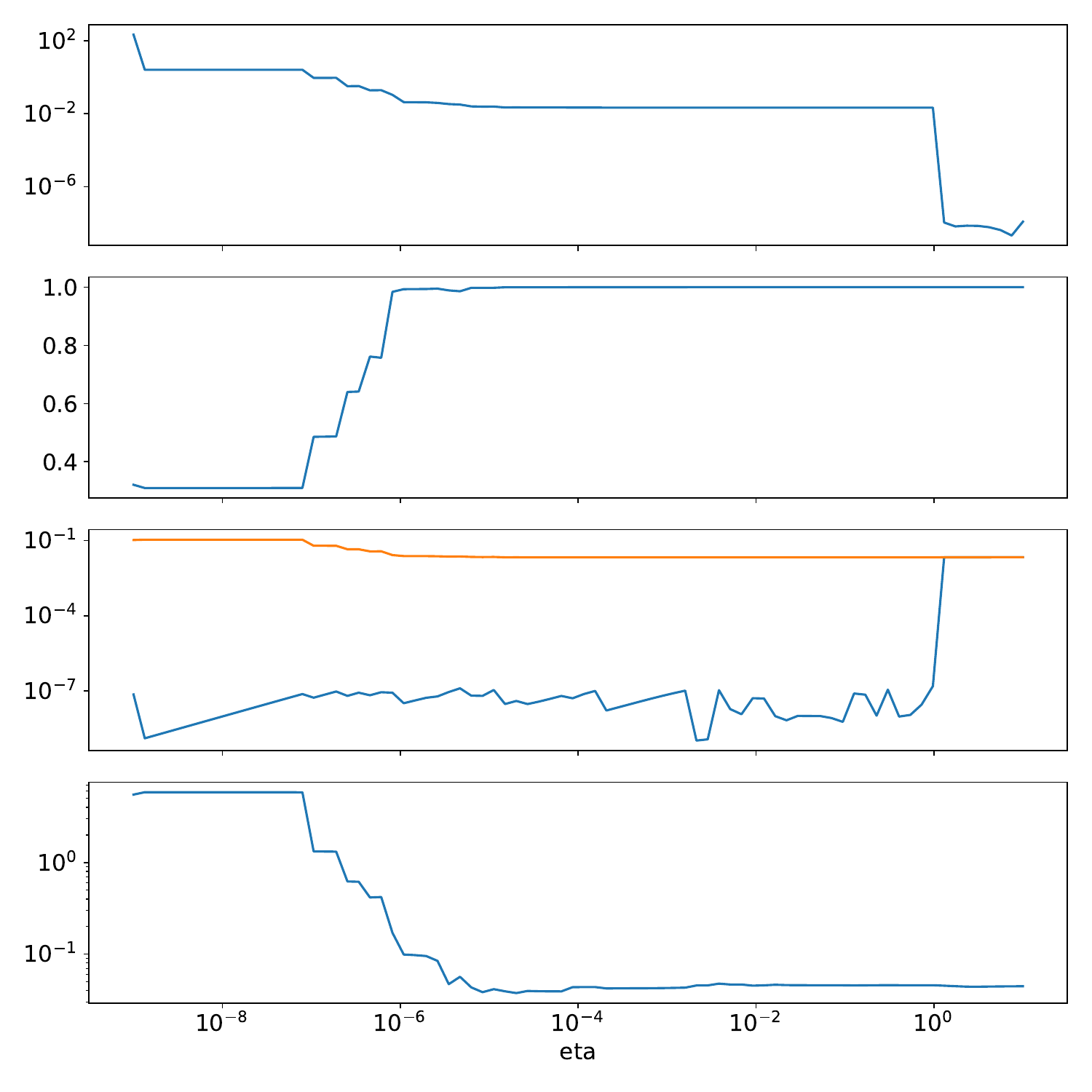}
    \caption{SDP relaxation with $80$ different values of $\eta$ for the Heidari truss with not-in-phase single harmonic load. The first figure shows the evolution of the trace difference $\tr\{\mathbf{X}-\mathbf{F}^*\mathbf{L}_{1,\omega}(\mathbf{a})^\dagger \mathbf{F}\}$ as $\eta$ ranged from $10^{-9}$ to $10$, the trace difference is indeed small as $\eta$ is large enough. The second figure shows the activation of mass constraint, the optimal structure uses all the available mass as $\eta$ is large enough. In the third figure, the blue (orange) curve shows respectively the evolution of $\theta$ (resp. the actual peak power computed at optimal design) as $\eta$ is increased. They agree only when $\eta$ is large enough. The last figure shows the evolution of the KKT residual defined in \eqref{eq:lp}.}
    \label{fig:many_eta}
\end{figure}

 For small values of $\eta$, Figure \ref{fig:many_eta} reveals that there is a nonzero trace difference at the optimal solution and the total mass constraint is not active. As a result, we observe in the early part of the graphs that the actual peak power of the optimal design is higher than that of the optimal design obtained for large $\eta$. With $\eta$ increasing, the trace equality becomes satisfied and the total mass constraint becomes active. The optimal value of $\theta$ then agrees with the actual peak power. However, the actual peak power does not improve further when the trace equality reaches $0$.

Finally, we try to estimate the optimality of the design obtained by SDP relaxations. The peak power of the truss under harmonic loads is a differentiable function of $a$ and the gradient can be obtained by sensitivity analysis as shown in Appendix \ref{appendix:d}, suggesting that it might also be possible to implement the gradient-based method to minimize the peak power. Without any additional variables, the peak power minimization can be expressed as
\begin{equation}\label{eq:nonlin}
    \begin{array}{ll}
         \underset{\mathbf{a}\in\mathbb{R}^n}{\min}& p(\mathbf{a})  \\
         \text{ s.t. }& \mathbf{a}_i\geq0, \mathbf{a}^T\mathbf{q}\leq m, 
    \end{array}
\end{equation}
where $p$ is the function of the peak power evaluated at $\mathbf{a}$, which takes $+\infty$ value if the load can not be carried at a design configuration $\mathbf{a}$. 
The formulation \eqref{eq:nonlin} makes peak power minimization a nonlinear optimization problem with linear inequality constraints. As a result, the linear constraint qualification (LCQ) holds, which implies that the Karush–Kuhn–Tucker conditions (KKT) are necessary for any local minimizer of the peak power minimization. In other words, if $\Bar{\mathbf{a}}$ is a local minimizer of the peak power minimization, then there exists a unique vector of Lagrange multipliers $\bm{\gamma}\in\mathbb{R}^n$ and $\Gamma\in\mathbb{R}$ such that the following KKT system is satisfied:
\begin{equation}
    \begin{cases}
        \nabla p(\Bar{\mathbf{a}})-\bm{\gamma} + \Gamma \bm{q} = 0,\\
        \gamma_i\geq 0, \Bar{a}_i\geq 0, \gamma_i\bar{a}_i = 0,\forall i,\\
        \Gamma \geq 0, m-\Bar{\mathbf{a}}^T\mathbf{q} \geq 0, \Gamma(m-\mathbf{a}^T\mathbf{q}) = 0.
    \end{cases}
\end{equation}

For any feasible $a$, we define the KKT residual at $\mathbf{a}$ as the solution of the linear programming problem
\begin{equation}\label{eq:lp}
    \begin{array}{ll}
         \underset{\bm{\gamma}\in\mathbb{R}^n,\Gamma\in\mathbb{R}}{\min}& \mathbf{a}^T\bm{\gamma}+\Gamma(m-\mathbf{a}^T\mathbf{q})  \\
         \text{ s.t. }& \begin{cases}
             \nabla p(\mathbf{a})-\bm{\gamma} + \Gamma \mathbf{q} = 0,\\
             \gamma_i\geq0,\forall i,\\
             \Gamma\geq 0.
         \end{cases} 
    \end{array}
\end{equation}
By the necessary condition of optimality, if $\bar{\mathbf{a}}$ is a local minimizer of \eqref{eq:nonlin}, then the solution of the linear programming problem \eqref{eq:lp} is the unique vector of Lagrange multipliers, and the optimal value is zero. If a design $\mathbf{a}$ is suboptimal but in the vicinity of a local optimal solution, then we may expect a low value of the KKT residual. The KKT residual can thus be used as an approximate optimality certificate even though it is not sufficient for non-convex problems.

For the illustrative problem considered in this section, we observe that the KKT residual is approximately $0.045$ for a sufficiently large $\eta$. Moreover, the KKT residual decreases with increasing $\eta$, as shown in the last subplot of Figure \ref{fig:many_eta}. This analysis suggests that the optimal solution with a sufficiently large $\eta$ might be indeed near to a local optimum of the peak power minimization problem \eqref{eq:nonlin}.

\subsection{Cantilever beam problem}\label{sec:cantilever}
Next, we consider a cantilever beam under a rotating tip load; see Figure \ref{fig:full_beam}. The design domain of the outer dimensions $2 \times 1$ is discretized into a uniform grid of $4\times 7$ nodes, with the left edge of the domain fixed. We connect each pair of the nodes by a bar element, resulting in $378$ elements in total. The rotating out-of-phase load is applied in the upper right corner of the beam. The components of this load act at the angular frequency of $15$ and 
$$
    \mathbf{c}_1(\mathbf{f})=-\begin{pmatrix}
        0\\
        \vdots\\
        i\\
        1\\
        \vdots\\
        0
    \end{pmatrix}/\sqrt{2}.
$$
The non-zeros components of $\mathbf{c}_1(\mathbf{f})$ correspond to the loaded degrees of freedom as shown in Figure \ref{fig:full_beam} and the load has an amplitude of $||\mathbf{c}_1(\mathbf{f})||=1$.

The truss members are made of a material with Young's modulus $E=25,000$ and density $\rho=1$. The total mass is bounded from above by $10$. In what follows, we consider convex relaxation with the penalty coefficient $\eta=10$, for which the minimization converges to the design shown in Figure \ref{fig:cantilever_pen_1}. This design contains $40$ finite elements and possesses the peak power of $0.147$ and the lowest angular eigenfreqency of $18.97$ which is strictly larger than the driving frequency of the loads $15$.

The optimization problem contained $389$ scalar design variables. The compiled problem instance had $3$ additional matrix variables containing $5295$ scalar variables, and there were $406$ constraints in total. The solution of the problem took $0.76$ seconds, suggesting a relatively good scalability of the method.

\begin{figure}[H]
    \centering
    \subfigure[]{
         \includegraphics[width=0.5\textwidth]{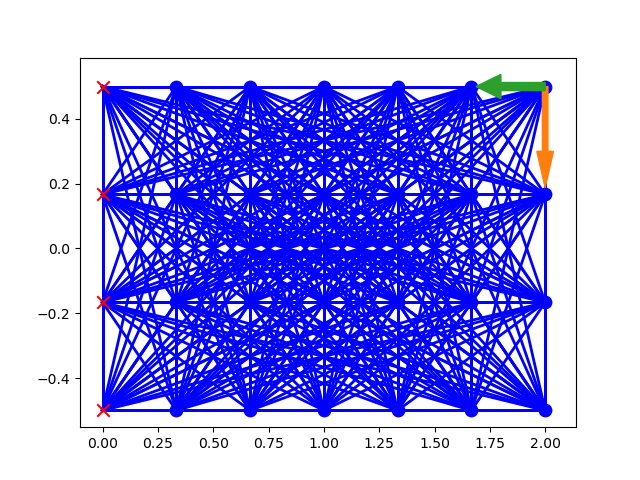}
        \label{fig:full_beam}
    }
    \hfill
    \subfigure[]{
         \includegraphics[width=0.5\textwidth]{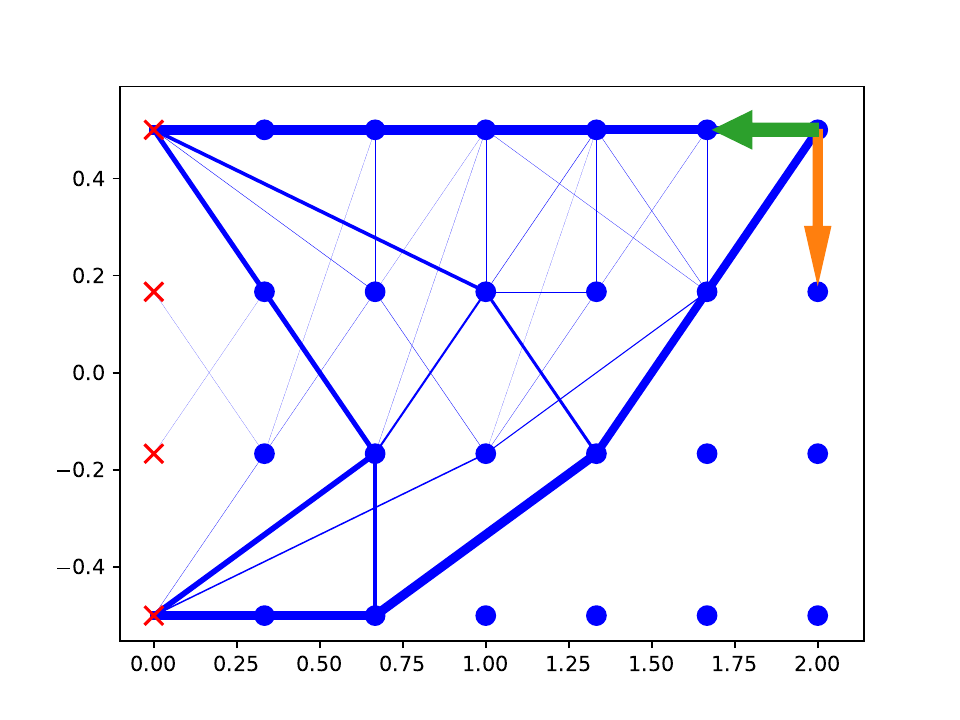}
        \label{fig:cantilever_pen_1}
    }
    \caption{Cantilever beam problem: (a) fully connected ground structure, and (b) optimal design for the penalized relaxation with $\eta=10$.}
\end{figure}

\begin{figure}[tb]
    \centering
    \includegraphics[width=0.5\textwidth]{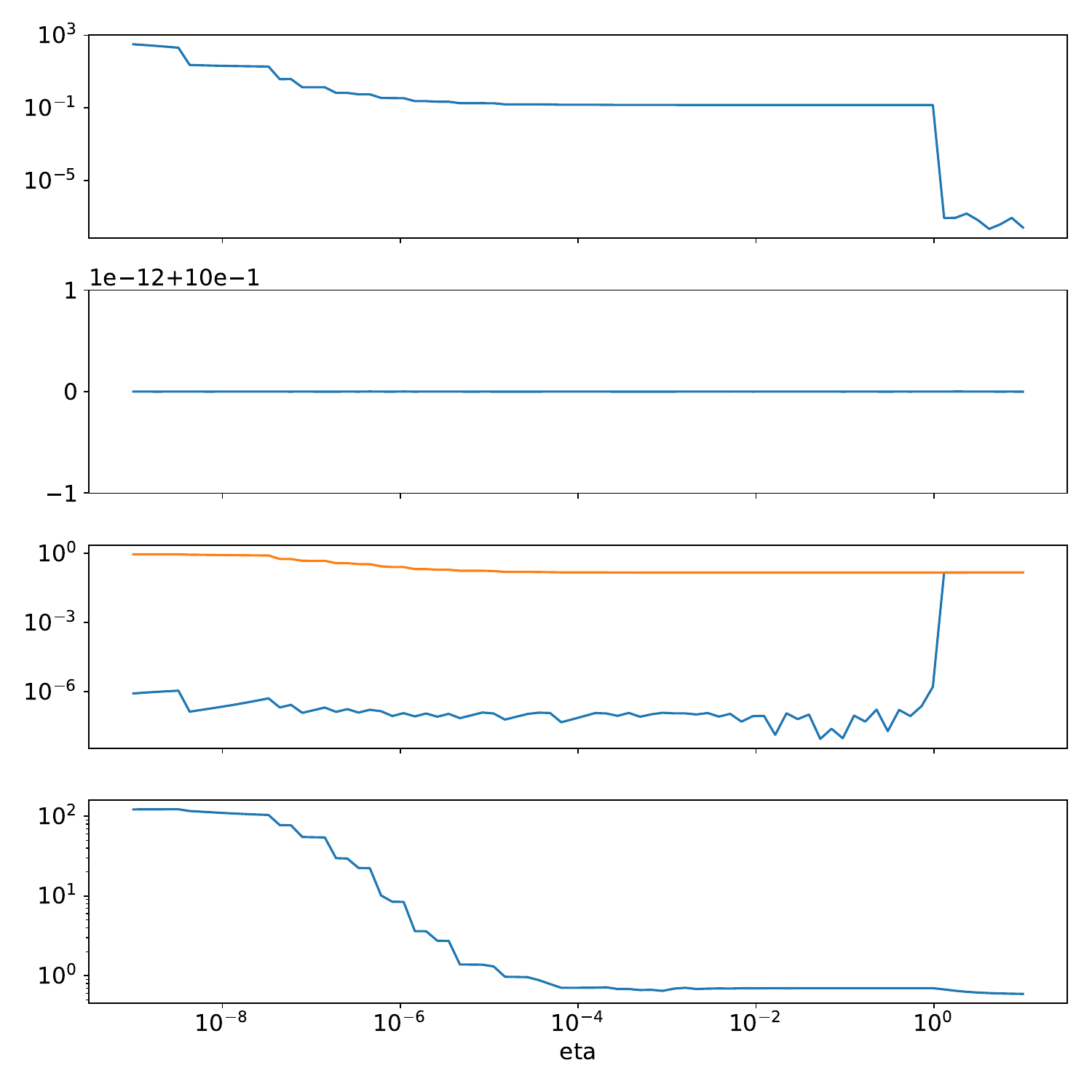}
    \caption{SDP relaxation with $80$ different values of $\eta$ for the cantilever problem. The first figure shows the evolution of the trace difference $\tr\{\mathbf{X}-\mathbf{F}^*\mathbf{L}_{1,\omega}(\mathbf{a})^\dagger \mathbf{F}\}$ as $\eta$ ranged from $10^{-9}$ to $10$, the trace difference is indeed small as $\eta$ is large enough. The second figure shows the activation of mass constraint, the optimal structure uses all the available mass as $\eta$ is large enough. In the third figure, the blue (orange) curve shows respectively the evolution of $\theta$ (resp. the actual peak power computed at optimal design) as $\eta$ is increased. They agree only when $\eta$ is large enough. The last figure shows the evolution of the KKT residual defined in \eqref{eq:lp}.}
    \label{fig:cantilever_many_eta}
\end{figure}

As in the previous example, we consider $80$ values of $\eta$ ranging from $10^{-9}$ to $10^1$ to visualize the effect of $\eta$ on the performance of the optimal design, see Figure \ref{fig:cantilever_many_eta}. From this figure we observe that the satisfaction of the trace difference becomes zero when $\eta$ is larger than $1$, and as a result, the optimal value of $\theta$ agrees with the actual value of the peak power. However, the mass bound is always active for each value of $\eta$ in this example. The KKT residual is reduced by $2$ orders of magnitude as $\eta$ increases.

\subsection{Peak power minimization under multifrequency rotating loads}\label{sec:two_rotation}

Optimizing structures under rotating loads is valuable for industrial applications. Here, we again consider the Heidari \textit{et al.}~\citep{heidari_optimization_2009} truss structure subjected to two rotating loads with different angular frequencies $\omega_1$ and $\omega_2$ simultaneously and the mass bound equal to $1$, Figure \ref{fig:two_rotation}. These loads are applied to the bottom nodes of the truss and share the same constant amplitude $1$. Their relative angle to the horizontal axis is a function of time and is equal to $\omega_1 t+\phi_1$ and $\omega_2t+\phi_2$, respectively.

\begin{figure}[H]
    \centering
    \includegraphics[width=0.5\textwidth]{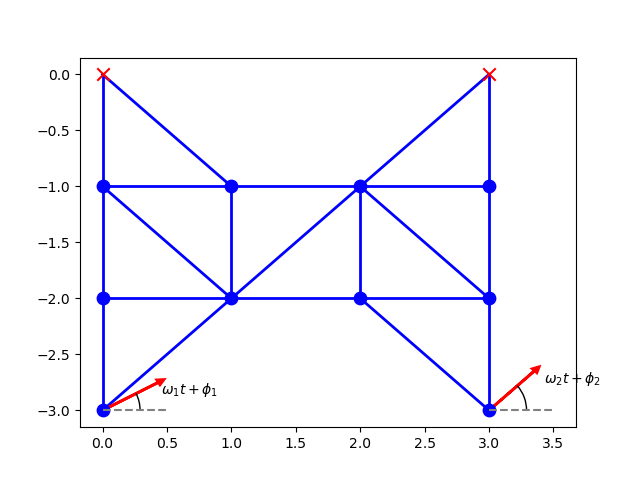}
    \caption{A truss subjected to two rotational loads of different angular frequencies $\omega_1$ and $\omega_2$.}
    \label{fig:two_rotation}
\end{figure}

In the context of the current method, we have to assume that two frequencies $\omega_1$ and $\omega_2$ are harmonic to each other, which implies that there is a basic frequency $\omega_0$ such that $\omega_1=n_1\omega_0$ and $\omega_2=n_2\omega_0$ are integer multiples of the basic frequency $\omega_0$. We assume without loss of generality that $n_2\geq n_1$, so that the time-varying load $f$ has $n_2$ Fourier components. To apply the SDP relaxation method we set the Fourier components of the load $\mathbf{f}$ to
\begin{equation}
    \forall n\geq 1, \mathbf{c}_n(\mathbf{f})=\frac{\delta_{n=n_1}e^{i\phi_1}}{2}\begin{pmatrix}
        0\\\vdots\\0\\\frac{1}{2}\\\frac{1}{2i}\\0\\0
    \end{pmatrix}+\frac{\delta_{n=n_2}e^{i\phi_2}}{2}\begin{pmatrix}
        0\\\vdots\\0\\0\\0\\\frac{1}{2}\\\frac{1}{2i}
    \end{pmatrix},
\end{equation}
where $\delta_{n=k}$ is the Kronecker delta notation such that $\delta_{n=k}$ is equal to one if and only if $n=k$.

For numerical experiments, we set $\phi_1=\frac{\pi}{2}$ and $\phi_2=-\frac{\pi}{2}$ and consider different values of $\omega_1$ and $\omega_2$, see Table \ref{table:two_rotation}. For each instance, we solve the resulting SDP relaxation with the penalty coefficient $\eta=10$ to make the trace difference approximately zero in the optimal solution. 

\begin{table*}[h]
    \centering
    \caption{Peak power minimization under multifrequency rotating loads. Parameters of the rotating loads and performance of the optimal design of the SDP relaxation with penalty. $\omega_1$ and $\omega_2$ denote the angular driving frequencies of the load and $\omega_0$ is the basic frequency such that $\omega_1=n_1 \omega_0$ and $\omega_2 = n_2 \omega_2$. In addition, $p[c(v)]$ denotes the peak power, and $\tilde{\omega}_1$ stands for the lowest resonance eigenfrequency. }
    \label{table:two_rotation}
        \begin{tabular}{|c|c|c|c|c|c|c|c|}
         \hline
         $\omega_1$ & $\omega_2$ & $\omega_0$ & $n_1$ & $n_2$ & $p[c(v)]$ & $\tilde{\omega}_1$ & KKT residual\\ 
         \hline
         7.5 & 15 & 7.5 & 1 & 2 & 0.0334 & 18.611&0.1448\\
         \hline
         10 & 15 & 5 & 2 & 3 & 0.0377 & 19.224&0.1630\\
         \hline
         12.5 & 15 & 2.5 & 5 & 6 & 0.0421 & 20.146&0.1815\\
         \hline
         13.125 & 15 & 1.875 & 7 & 8 & 0.0432 & 20.436&0.1882\\
         \hline
        \end{tabular}
\end{table*}
In each instance, we reached a sub-optimal solution with non zero KKT residual of the peak power and the mass bound was active. Because there is no comparable design available in the literature, we further compare the objective function value with that of the common initial point, the uniform truss design. Compared to this case, each of our optimized designs provides approximately $50\%$ improvement in terms of peak power. Each of the optimized designs also has its first resonance frequency strictly higher than the driving frequency $15$. The optimal designs for $n_2=2$ and $n_2=8$ are shown in Figures \ref{fig:two_rot_2} and \ref{fig:two_rot_8}.

\begin{figure}[H]
    \centering
    \subfigure[]{
         \includegraphics[width=0.45\textwidth]{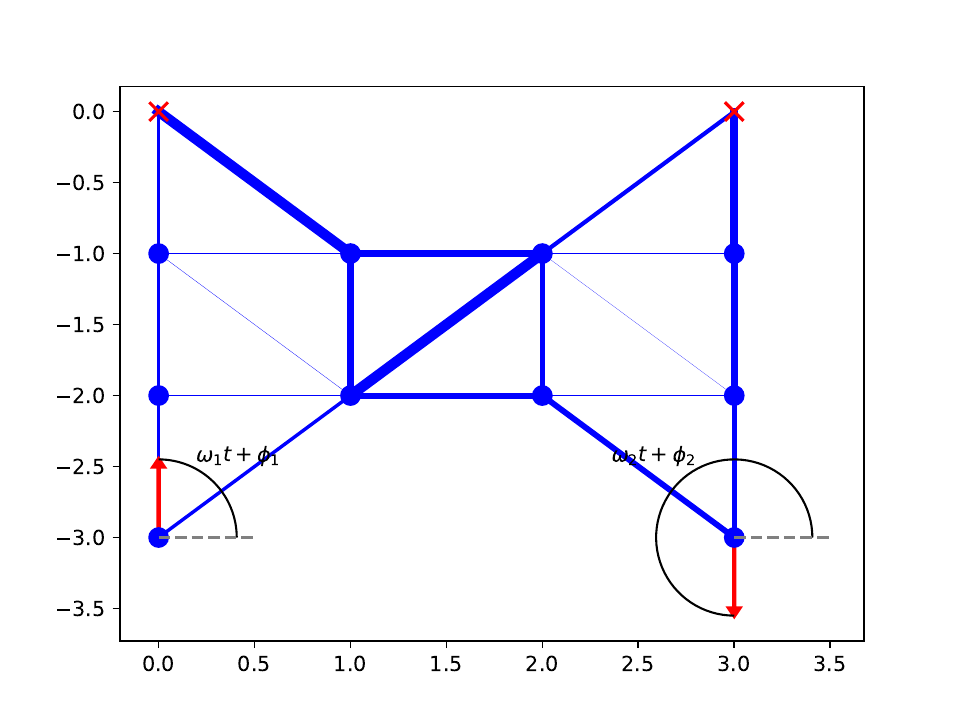}
        \label{fig:two_rot_2}
    }
    \hfill
    \subfigure[]{
         \includegraphics[width=0.45\textwidth]{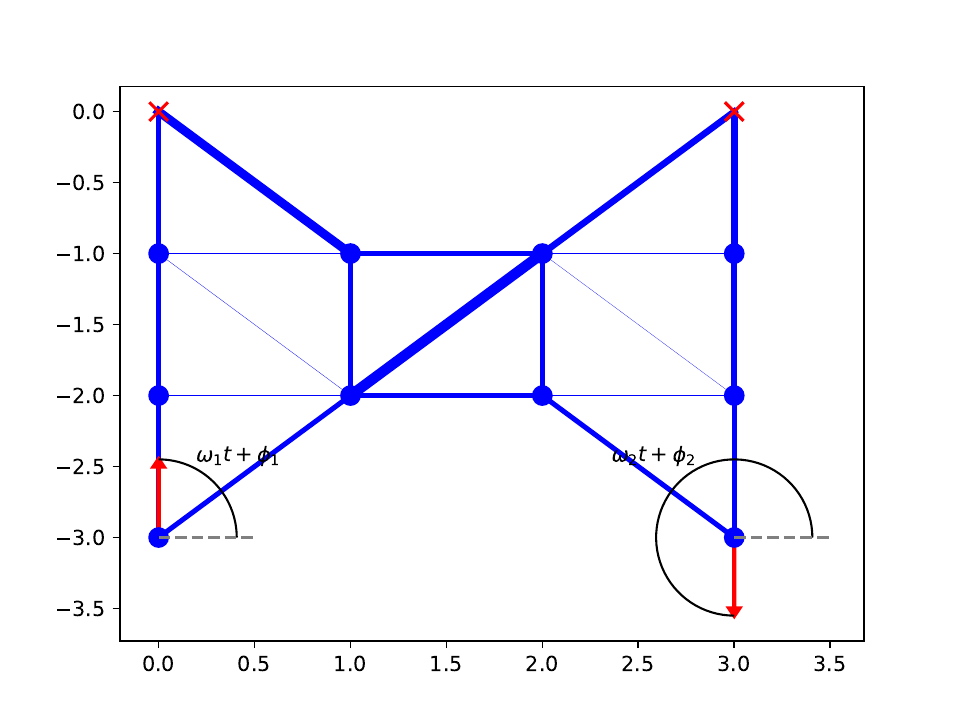}
        \label{fig:two_rot_8}
    }
    \caption{Peak power minimization under multifrequency rotating loads: (a) optimized design for $n_2=2$, and (b) for $n_2=8$.}
\end{figure}

\begin{table*}[h]
    \centering
    \caption{Problem sizes of the peak power minimization under multifrequency rotating loads. $\omega_1$ and $\omega_2$ denote the angular driving frequencies of the load and $\omega_0$ is the basic frequency such that $\omega_1=n_1 \omega_0$ and $\omega_2 = n_2 \omega_2$. In addition, $n_\mathrm{var}$ and $n_\mathrm{constr}$ denote the number of variables and constraints, respectively, and $t$ stands for the solution time.}
    \label{table:two_rotation_metrics}
    
    \begin{tabular}{|c|c|c|c|c|c|c|c|}
         \hline
         $\omega_1$ & $\omega_2$ & $\omega_0$ & $n_1$ & $n_2$ & $n_\mathrm{var}$ &$n_\mathrm{constr}$ & $t$ (s)\\ 
         \hline
         7.5 & 15 & 7.5 & 1 & 2 & 108&10 &0.094\\
         \hline
         10 & 15 & 5 & 2 & 3 & 201&14 &0.163\\
         \hline
         12.5 & 15 & 2.5 & 5 & 6 & 684& 26&1.293\\
         \hline
         13.125 & 15 & 1.875 & 7 & 8 & 1176 &34&2.567\\
         \hline
    \end{tabular}
\end{table*}
     
The size of the corresponding SDP problem for each instance is reported in Table \ref{table:two_rotation_metrics}, but the actual size of the solved problem may differ due to the remodeling done in the CVXPY package, which transforms the problem into a conic form accepted by MOSEK. This remodeling involves additional variables and constraints. For example, every LMI constraint $\mathbf{L(x)}\succeq 0$ is converted to $\mathbf{L(x)=X}$ and $\mathbf{X}\succeq 0$ with a slack variable $\mathbf{X}$ that has the same size of the LMI $\mathbf{L(x)}$ and additional linear equality constraints are introduced to ensure equality $\mathbf{L(x)=X}$. Moreover, to convert a Hermitian PSD constraint into a real symmetric PSD constraint, one also needs to include additional variables and constraints.

As a result, even when the number of ``active'' Fourier components was the same throughout all test cases, the sizes of the problems were different. Typically, for the last instance where $n_2=8$, the constraint $\begin{pmatrix}
    \mathbf{X} & \mathbf{F}^*\\
    \mathbf{F} & \mathbf{L}_{n_2,\omega_0}(\mathbf{a})
\end{pmatrix}\succeq 0$ involves a LMI of the size $8\times(21+3)=192$. To transform it into the standard formulation of MOSEK, we need to include a $384$ by $384$ real symmetric matrix variable and additional linear equality constraints. From this we observe that it is definitely needed to exploit sparsity to solve for large instances of the peak power minimization using the proposed SDP method.

\subsection{Peak power minimization under multiple frequency loads}\label{sec:multi_freq}
The convex relaxation method is also capable of optimizing structural designs under periodic loads with multiple frequency components. The theoretical construction of the minimization problem remains the same as in the single-frequency situation and is summarized in Appendix \ref{appendix:peak_power_minimization_multi}. 

A finite energy (being $L^2$ function) periodic time-varying load $\mathbf{f}$ can be decomposed into its Fourier series. In theory, if we truncate the Fourier sequence of a periodic load $\mathbf{f}$ up to a finite number of Fourier coefficients and optimize the truss using the presented method with respect to the truncated load, it approximates the optimization with respect to the original periodic load. However, not every such periodic load is suitable for the method presented here. 

For example, let us again consider the Heidari \textit{et al.} \citep{heidari_optimization_2009} problem that we have already presented in Section \ref{sec:examples_heidari}, but with a modification of the dynamic loads. In particular, we apply the loads shown in Figure \ref{fig:heidari_bc_multi}, where $f_1$ and $f_2$ are two periodic loads acting simultaneously. 
\begin{figure}[H]
\centering
\includegraphics{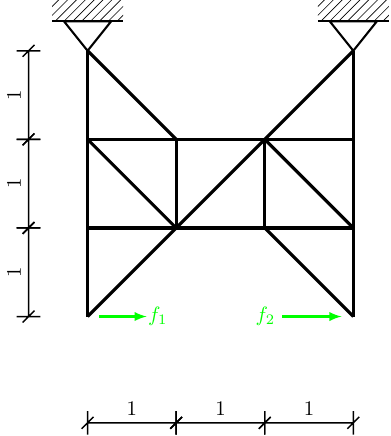}
\caption{Ground structure of the $21$-element truss structure under multiple frequency loads $f_1$ and $f_2$.}\label{fig:heidari_bc_multi}
\end{figure}
$f_1$ is a periodic function with period $T$ and amplitude $1$ such that
\begin{equation}
    f_1(t)=\begin{cases}
        -1\text{ if }\frac{-T}{2}< t\leq0,\\
        1\text{ if }0< t\leq \frac{T}{2}.
    \end{cases}
\end{equation}
The periodic load $f_2$ has the same amplitude and period with a delay $T_0$, so that $\forall t\in\mathbb{R},f_2(t)=f_1(t-T_0)$. To approximate the nodal velocity $v(t)$ for a given vector of cross section areas $a$, we solve for the following linear equations up to a finite number $N$ of Fourier coefficients:
\begin{equation}
    \mathbf{K}_{k\omega}(\mathbf{a})\mathbf{c}_k(\mathbf{v})=ik\omega \mathbf{c}_k(\mathbf{f}), \forall |k|\leq N.
\end{equation}
Because the period is $T$, the lowest harmonic frequency is $\omega=\frac{2\pi}{T}$. We calculate the Fourier coefficients of $f_1$ as
\begin{equation}
    c_k(f_1)=\begin{cases}
        0 \text{ if } k=0,\\
        \frac{1}{T}\int_{-T/2}^{T/2}f_1(t)e^{-ik\omega t}\de t=\frac{i}{\pi}\frac{(-1)^k-1}{k}\text{ otherwise.}
    \end{cases}
\end{equation}
Because $f_2$ is delayed in time, $\forall k,c_k(f_2)=e^{-ik\omega T_0}c_k(f_1)$. Finally, the overall Fourier coefficients of $\mathbf{f}(t)$ read as
\begin{equation}
    \forall k\neq 0,\mathbf c _k(\mathbf f) = \begin{pmatrix}\vdots\\ c_k(f_1)\\0\\c_k(f_2)\\0\end{pmatrix}
\end{equation}

Let $T=2(s)$ and $T_0=0.2(s)$ and consider a truss with uniform cross-section areas of the total mass $m=1$. For this case, we calculate the power function with $\mathbf{f}$ using a truncated Fourier series. In Figure \ref{fig:power_functions}, we plot the instant power of the truss in time for different numbers of Fourier coefficients, observing that the power function requires at least $15$ harmonics. Their spectrum in Figure \ref{fig:norm_of_ckv} reveals that the coefficients must cover all resonance frequencies (drawn in red). For the adopted discretization, the maximum eigenfrequency is evaluated as $233$ rad/s, so we need to include at least $\frac{233}{\omega}\approx 74$ Fourier components of $\mathbf{f}$ to cover the range of eigenfrequencies of the uniform truss.

\begin{figure}[H]
    \subfigure[]{
        \centering
         \includegraphics[width=0.5\textwidth]{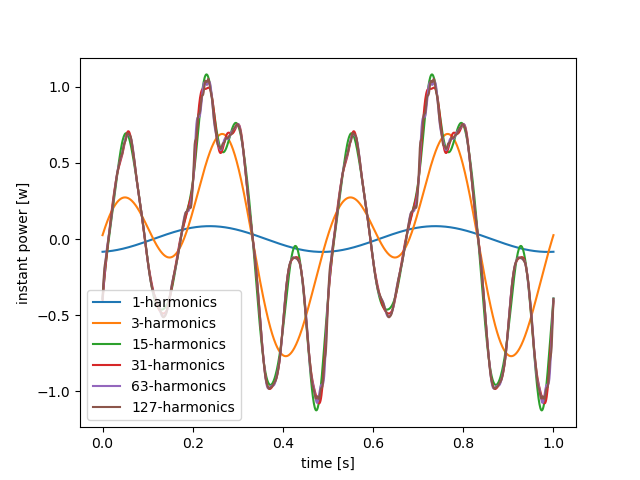}
        \label{fig:power_functions}
    }
    \subfigure[]{
        \centering
        \includegraphics[width=0.5\textwidth]{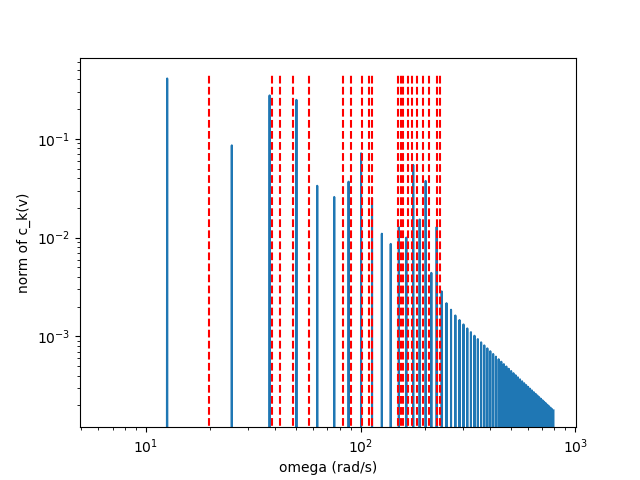}
        \label{fig:norm_of_ckv}
    }
    \caption{Response of the Heidari \textit{et al.} truss under time-varying load $\mathbf{f}$: (a) Effect of truncations of the time-varying load on the instant power, and (b) norm of the Fourier coefficients of the nodal velocity up to some finite number of frequencies in logarithmic scale. The vertical red lines indicates the position of the eigenfrequencies of the specific design. It is observed that the Fourier components of the velocity start to converge only when all the eigenfrequencies are covered which shows that for the specific time varying load it is required to include a large number of Fourier components to compute an accurate response of the structure.}
\end{figure}

Because we have assumed that the maximum frequency of the load is less than the minimum resonance frequency $\mathbf{K}_{N\omega}(\mathbf{a})\succeq 0$, and because the minimum resonance frequency of a truss under a mass constraint cannot be arbitrarily large, we can only consider small enough $N$. This limitation suggests that the current method of peak power minimization is suitable for time-varying loads of low frequencies whose Fourier components have a fast decreasing norm. This implies smooth-enough time-varying loads $\mathbf{f}$. If the time-varying loads $\mathbf{f}$ are functions with $C^k$ regularity ($k$ times derivable and with a continuous $k$-th derivative), then the norm of $\mathbf{c}_n(\mathbf{f})$ tends to zero at speed $\frac{1}{n^k}$. When $\mathbf{f}$ is smooth, we can expect that a few numbers of $\mathbf{c}_n(\mathbf{f})$ suffice to approximate the original $\mathbf{f}$. For the current example, since $\mathbf{f}$ is not smooth, it requires many Fourier components.

For the reasons outlined in the previous paragraph, in the remaining part of this section, we solve the peak power minimization problem while considering truncation of $f$ up to the third Fourier coefficient, i.e.,
\begin{equation}
    \mathbf{f}(t)=\sum_{k=-3}^3\mathbf{c}_k(\mathbf{f})e^{ik\omega t}.
\end{equation}
We lower its frequency, thus, we set $T=2(s)$ for numerical reasons. The convex relaxation with penalty of the corresponding peak power minimization can be written as the minimization problem \eqref{eqn:multiple_harmonics}  with $N=3$. Taking $\eta=10$, set so that the mass constraint in relaxation \eqref{eqn:multiple_harmonics} is active and the matrix inequality $\mathbf{X}=\mathbf{F}^*\mathbf{L}_{3,\omega}(\mathbf{a})^\dagger \mathbf{F}$ holds, the minimization converges to the design in Figure \ref{fig:opt_truss} that exhibits the instant power function shown in Figure \ref{fig:opt_power_function}. The peak power is improved by $40\%$ compared to the uniform truss. The lowest resonance frequency of the optimal truss is evaluated as $13.062$ rad/s, which is less than the highest frequency of the load $3\omega=\frac{3\times 2\pi}{T}=3\pi$ rad/s.

\begin{figure}[H]
    \centering
    \subfigure[]{
        \includegraphics[width=0.5\textwidth]{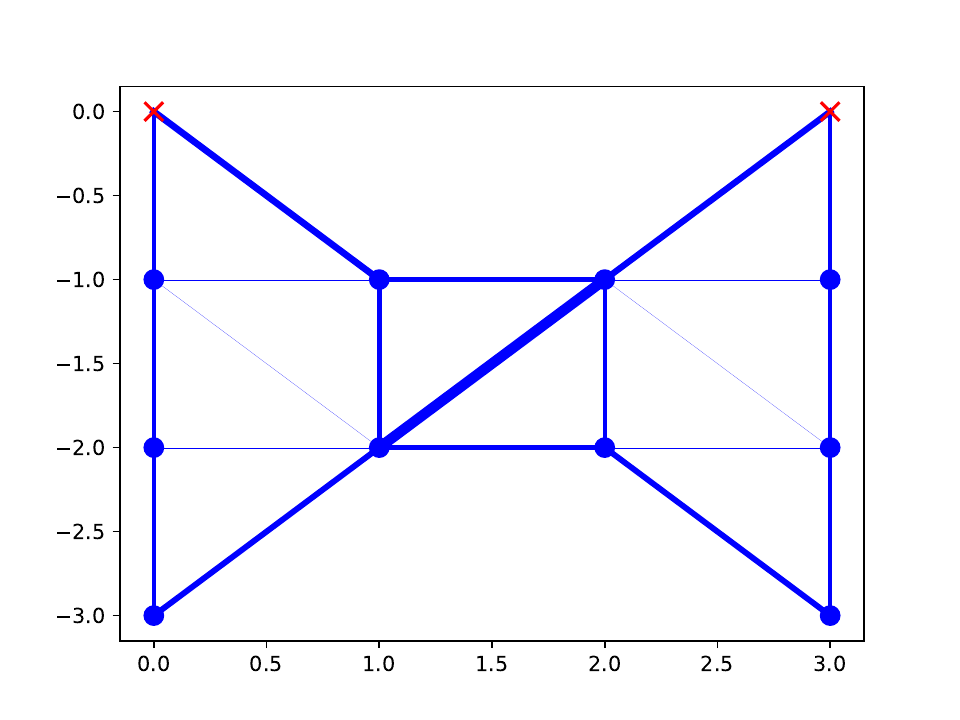}
        \label{fig:opt_truss}
    }
    \subfigure[]{
         \includegraphics[width=0.5\textwidth]{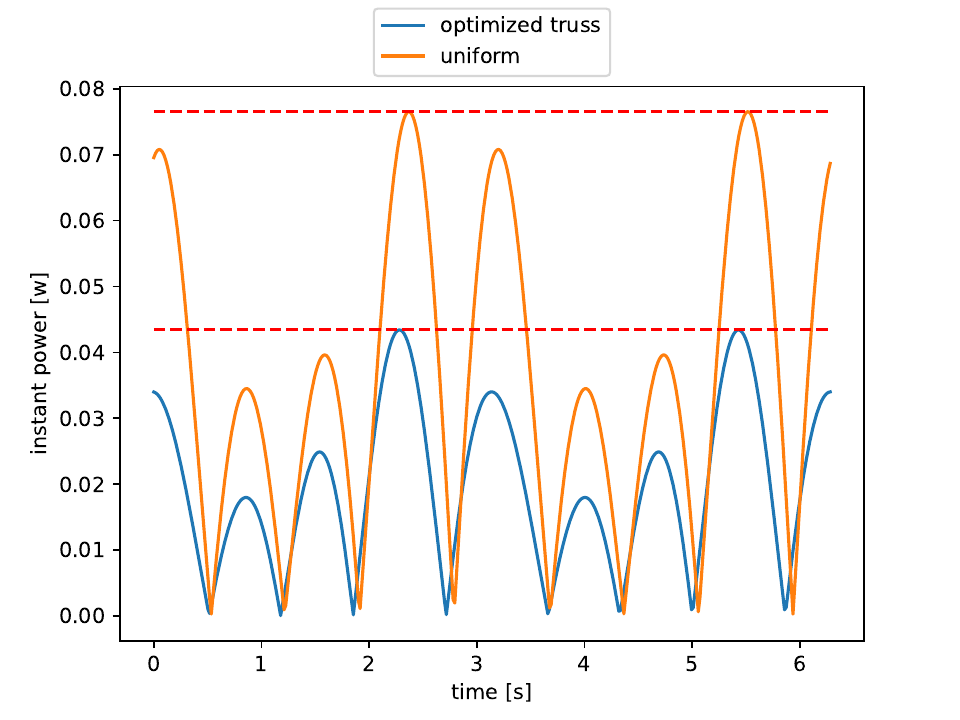}
        \label{fig:opt_power_function}
    }
    \caption{Optimized design of the Heidari \textit{et al.} truss obtained by solving a convex relaxation with penalty $\eta=10$ and $N=3$: (a) The optimized design , and (b) Comparison between the instant power function of the uniform truss and the optimized truss under harmonic loads with $3$ Fourier components.}
\end{figure}

Similar results can be obtained by including a larger number of Fourier components of $\mathbf{f}$. For $N=5$, the matrix $\mathbf{L}_{5,\omega}(\mathbf{a})$ increases in size compared to $\mathbf{L}_{3,\omega}(\mathbf{a})$ due to the larger matrix variables $\mathbf{X}$, $\mathbf{Q}_1$ and $\mathbf{Q}_2$ of the respective sizes $15\times15$, $10\times 10$, and $10\times 10$. Following the former setup, we also adopt the penalty factor $\eta=10$.

\begin{figure}[H]
    \centering
    \subfigure[]{
        \includegraphics[width=0.5\textwidth]{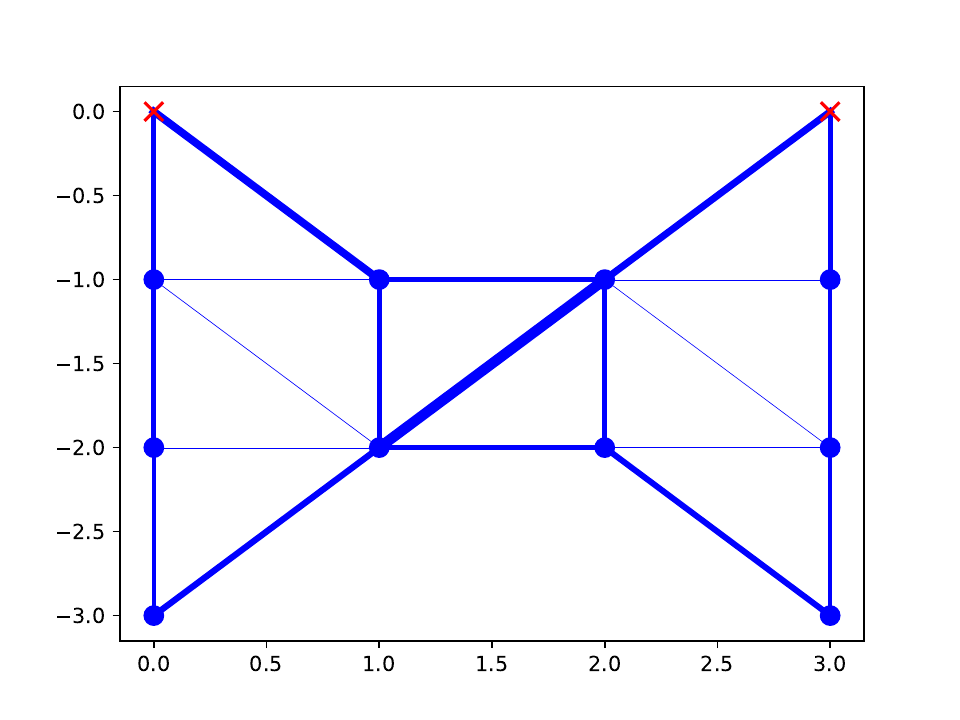}
        \label{fig:opt_truss_5}
    }
    \subfigure[]{
         \includegraphics[width=0.5\textwidth]{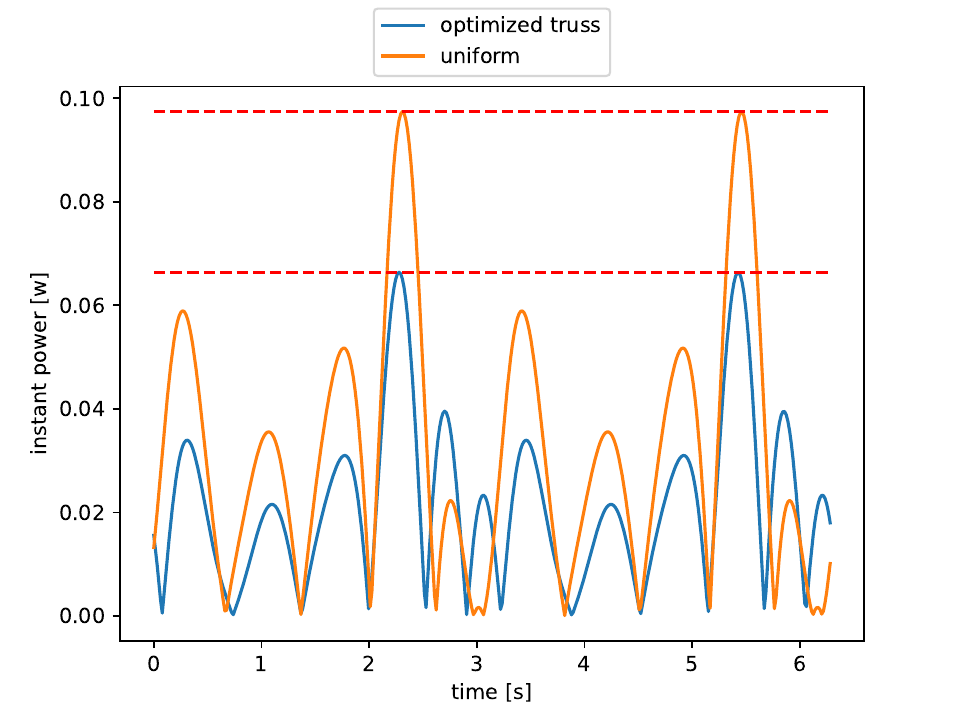}
        \label{fig:opt_power_function_5}
    }
    \caption{Optimized design of the Heidari \textit{et al.} truss obtained by solving a convex relaxation with penalty $\eta=10$ and $N=5$: (a) The optimized design, and (b) Comparison of the instant power function of the uniform truss and the optimized truss under harmonic loads with $5$ Fourier components.}
\end{figure}

Using the proposed relaxation with penalization, we reached the optimized design shown in Figure \ref{fig:opt_truss_5}. This optimized design exhibits the instant power shown in Figure \ref{fig:opt_power_function_5}. Compared to the uniform truss, it succeeded in suppressing the peak power by approximately $40\%$; an improvement similar to that achieved for the three harmonics. However, we reached a considerably increased lowest free-vibration eigenfrequency of the value $23.950$ rad/s, which is in fact strictly higher than the maximum frequency of the load, $5\pi\approx 15.7$ rad/s.

Clearly, we have observed a certain similarity between the optimal design for $N=3$ and $N=5$. To compare them, we calculate the peak power of the different designs when loaded with $\mathbf{f}$ truncated up to $3$, $5$ and $31$ Fourier components, which is sufficiently large to approximate the peak power of the original periodic load; see Table \ref{tab:compare}.

\begin{table*}[!htbp]
\centering
    \caption{Influence of the number of conserved harmonics on the optimized design for the peak power minimization under multiple frequency loads. In the current table, the first column indicates the number of Fourier components that are kept in the period load to evaluate the peak power. Starting from the second column, we present the peak power performance of the uniform truss and the optimized trusses obtained by minimizing peak power while keeping $N=3$ or $5$ Fourier components of the original periodic load. }
    \label{tab:compare}
	\begin{tabular}{|c|c|c|c|}
		\hline
		$N$ & Uniform truss & Optimized for $N=3$ & Optimized for $N=5$ \\
        \hline
        3 & 0.0765& 0.0434 &0.0451 \\
        \hline
        5 & 0.0974& 0.0674 &0.0664 \\
        \hline
        31 & 1.631& 14.887 &0.8694 \\
        \hline
	\end{tabular}
\end{table*}

In Table \ref{tab:compare}, design $N=3$ and $N=5$ provide the best performance for loads with $3$ and $5$ Fourier components. This is not surprising since they are optimized with the number of Fourier components retained in $\mathbf{f}$. When loaded with $\mathbf{f}$ where we retain a different number of Fourier components, they still outperform the uniform truss in terms of peak power. However, the effect of optimization is not clear when they are loaded by $\mathbf{f}$ with $31$ Fourier components. The performance of the design $N=3$ is worse and the design $N=5$ is nearly $50\%$ better than the uniform truss.

To explain this observation, we first computed the lowest eigenfrequencies of the two optimal designs. Furthermore, we also evaluated the smallest distance of each eigenfrequency to the multiples of $\omega$ and repeated this computation for the uniform truss; see Table \ref{tab:compare_spectrum}.

\begin{table*}[!htbp]
\centering
\caption{Spectral properties of various designs for the peak power minimization under multiple frequency loads. Each row represents a different design. The first two rows are the designed obtained by minimizing peak power while keeping $N=3$ and $5$ Fourier components of the original load. The last row is the uniform truss. For each column, we present the $n$-th eigenfrequency of the design and also the smallest distance to a multiple of the basic frequency $\omega_0$.} $N$ denotes the number of harmonics we used to approximate the time-varying load in the optimization problem, $\tilde{\omega}_n$ denotes the $n$-th angular eigenfrequency of the design, $\omega_0$ is the basic frequency, and $\ell\in \mathbb{N}$ is such that $\lvert\tilde{\omega}_n - \ell \omega_0\rvert$ is minimal.
    \label{tab:compare_spectrum}
	\begin{tabular}{|c|c|c|c|}
		\hline
		& $(\tilde{\omega}_1,\lvert\tilde{\omega}_1-\ell\omega_0\rvert)$ & $(\tilde{\omega}_2,\lvert\tilde{\omega}_2-\ell\omega_0\rvert)$ & $(\tilde{\omega}_3,\lvert\tilde{\omega}_3-\ell\omega_0\rvert)$ \\
        \hline
        Design for $N=3$ & $(13.063,~0.496$)& $(17.873,~0.976)$ & $(21.455,~0.537)$ \\
        \hline
        Design for $N=5$ & $(23.950,~1.183)$& $(26.272,~1.140)$ & $(33.805,~0.753)$ \\
        \hline
        Uniform truss & $(20.244,~1.395)$ & $(41.838,~0.998)$ & $(47.643,~0.519)$\\
        \hline
	\end{tabular}
\end{table*}

The optimal design with $N=3$ has the lowest first eigenfrequency; furthermore, the first three eigenfrequencies of $N=3$ tend to be closer to the frequencies present in the periodic load $\mathbf{f}$. As a result, there is stronger effect of resonance, which results in a worsened performance of peak power when the design is under the time varying load with many frequencies components. This numerical experiments suggests once again that the method we developed in this study is more adapted for optimizing a structure when it is under time-varying load with a small number of harmonic frequencies.

\section{Summary}\label{sec:discussion}

In this study, we have investigated the minimization of the peak power of truss structures under multiple harmonic loads whose driving frequencies are below the lowest resonance frequency of the structure itself. Starting from a general problem formulation, we have exploited the semidefinite representability of the peak power function under the equilibrium condition by exploiting the SOS positivity certificate of trigonometric polynomials. This allows us to extend the results of \citep{heidari_optimization_2009} to multiple harmonic out-of-phase loads that are decomposed into Fourier series. With this, we have developed an equivalent but generally nonconvex formulation that reduces to a convex setting for in-phase single-frequency loads.

For the general settings, convex reformulation is no longer possible. To address the nonconvex problem, we have first introduced a Lagrange-type relaxation by moving the non-convex constraint $\tr\{\mathbf{X}-\mathbf{F}^T\mathbf{L(a)}^\dagger \mathbf{F}\}\leq 0$---which is a difference of convex functions---into the objective and penalized its violation by a positive penalization factor $\eta$. Finally, we neglected the concave term $-\eta\tr\{\mathbf{F}^T\mathbf{L(a)}^\dagger \mathbf{F}\}$ and derived a convex relaxation. With $\eta$ large enough, the solution of the convex relaxation is a feasible point for the original non-convex program and hence is its suboptimal solution.

These theoretical results have been numerically illustrated with four examples. These illustrations revealed that the method indeed reduces to the case of \citep{heidari_optimization_2009} for single-frequency in-phase loads, and for more general cases the method converges to suboptimal points with small KKT residuals, thus providing almost locally optimal solutions.

Furthermore, the numerical examples also uncovered directions for potential future improvements. First, the frequency components of the load must be integer multiples of a basic frequency. Thus, if the basic frequency is small and the driving frequencies of the loads are high, the number of considered frequency components must be large. This also suggests that exploiting the sparsity \citep{kim_exploiting_2011,kocvara_decomposition_2021} in the constraint $$\begin{pmatrix}
    \mathbf{X} & \mathbf{F}^T \\ \mathbf{F} & \mathbf{L(a)}
\end{pmatrix}\succeq 0$$ would substantially help with solution of large-scale instances. Finally, the last example has revealed that the presented method applies to loads with frequencies below the first resonance, and not all periodic loads satisfy such assumption. When some of the frequencies components of the load are higher than the first resonance frequency, the peak power minimization remains an open question.

\section{Replication of results}
The \verb|Python| code for reproducing the experiments in section \ref{sec:examples} is avaible in the repository : \url{https://gitlab.com/ma.shenyuan/truss_exp}.

A \verb|conda| environment file is provided in the repository for ensuring that the package dependencies are met and for reproducing the environment in which the experiments were conducted by the authors.

\begin{appendices}
\section{Semidefinite representability of peak power under harmonic loads}\label{appendix:A}
\subsection{Nonnegativity certificate of trigonometric polynomial}
Readers can find the main results of this subsection in the book \citep{dumitrescu_positive_2017}, which covers applications in signal processing such as filter design. We are interested in the so-called trigonometric polynomials in what follows.
\begin{definition}[Trigonometric polynomial]
A trigonometric polynomial $f$ is a real-valued function with a variable in $\mathbb{T}$. It takes the form:
\begin{equation}
	f(z)=\sum_k c_k z^{-k}
\end{equation}
with finitely many non-zero coefficients $c_k$ with $k\in\mathbb{Z}$. Being real valued for all $z\in\mathbb{T}$, the sequence of coefficients must satisfy the symmetry condition $\forall k\in\mathbb{Z}, c_{-k}=c_k^*$. The degree of $f$ is the largest index $k$ of the non-zero coefficient $c_k$.
\end{definition}

A \textit{nonnegative} trigonometric polynomial is one that satisfies $f(z)\geq 0, \forall z\in\mathbb{T}$. We want to obtain a tractable characterization of the nonnegativity of a trigonometric polynomial. For univariate trigonometric polynomials, a powerful characterization is available using sum-of-squares (SOS).
\begin{definition}[SOS trigonometric polynomial]
A SOS trigonometric polynomial $f$ is such that there exists a finite family of complex polynomials $\{h_j(z)=\sum_{k=0}^N h_{jk}z^k\}_{j=1,\dots,r}$ satisfying
\begin{equation}
	\forall z, f(z)=\sum_{j=1}^{r} h_j(z)h_j(z)^*.
\end{equation}
\end{definition}
Being SOS, such a polynomial is trivially nonnegative. However, the converse is true for univariate trigonometric polynomials only; we refer the reader to \citep[Theorem 1.1]{dumitrescu_positive_2017}. The converse is false for trigonometric polynomials with more than $2$ variables. The proof relies on a careful study of the position of roots of a univariate polynomial.
\begin{theorem}[SOS certificate of nonnegativity]\label{thrm:pos_sos}
A univariate trigonometric polynomial is nonnegative \textit{iff} it is SOS.
\end{theorem}

The next theorem provides a link between the SOS characterization and PSD matrices. Testing if a trigonometric polynomial is SOS has been shown to be equivalent to the existence of a PSD matrix representation of the polynomial. Consider the following vector $\bm{\psi}$
\begin{equation}
	\bm{\psi}(z)=(1,z,\dots,z^{N})^T
\end{equation}
containing all the canonical basis polynomials up to degree $N$. A trigonometric polynomial has a so-called Gram matrix representation.
\begin{definition}
	Consider a degree $N$ trigonometric polynomial $f$. A Hermitian matrix $\mathbf{Q}\in\mathbb{S}^{N+1}$ is a Gram matrix of $f$ if \begin{equation}
		\forall z, f(z)=\bm{\psi}(z)^*\mathbf{Q}\bm{\psi}(z).
	\end{equation}
\end{definition}
There exists a family of constant matrices that establishes a link between $\mathbf{Q}$ and the coefficients of $f$. To see this, we note that
\begin{equation}
	\bm{\psi}(z)^*\mathbf{Q}\bm{\psi}(z)=\tr\left\{\bm{\psi}(z)\bm{\psi}(z)^*\mathbf{Q}\right\},
\end{equation}
where
\begin{equation}
	\bm{\psi}(z)\bm{\psi}(z)^*=
		\begin{pmatrix}
			1&z^{-1}&z^{-2}&\dots&z^{-N}\\
			z&1&z^{-1}&\ddots\\
			z^2&z&1&\ddots\\
			\vdots&\ddots&\ddots&\ddots&z^{-1}\\
			z^N&&&z&1
		\end{pmatrix}.
\end{equation}
By defining matrices $\mathbf{\Lambda}_k$ with Toeplitz structure,
\begin{equation}
	(\mathbf{\Lambda}_k)_{i,j}=1\iff j-i=k
\end{equation}
and $(\mathbf{\Lambda}_k)_{i,j}=0$ otherwise, we see that $\bm{\psi}(z)\bm{\psi}(z)^*=\sum_{k=-N}^{N}\mathbf{\Lambda}_kz^{-k}$. Thus, by identifying the coefficients, we also observe that
\begin{equation}\label{appendix:sos_def}
    \begin{split}
        &\forall z,f(z)=\sum_kc_kz^{-k}=\sum_k\tr{\mathbf{\Lambda}_k\mathbf{Q}}z^{-k}\\
    \iff&\forall k, c_k=\tr{\mathbf{\Lambda}_k\mathbf{Q}}.
    \end{split}
\end{equation}
The Gram matrix $\mathbf{Q}$ of a trigonometric polynomial is not unique. 
\begin{example}\label{example:gram}
	Let us now consider the degree-$2$ trigonometric polynomial $$f(z)=c_{-2}z^2+c_{-1}z+c_0+c_1z^{-1}+c_2z^{-2}$$ and we search a Gram matrix representation $\mathbf{Q}=\begin{pmatrix}
		q_{00}&q_{01}&q_{02}\\
		\overline{q}_{01}&q_{11}&q_{12}\\
		\overline{q}_{02}&\overline{q}_{12}&q_{22}
	\end{pmatrix}$. $\mathbf{Q}$ must satisfy:
	\begin{equation}
		\begin{split}
			&k=-2:\tr\{\mathbf{\Lambda}_{-2}\mathbf{Q}\}=q_{02}=c_{-2},\\
			&k=-1:\tr\{\mathbf{\Lambda}_{-1}\mathbf{Q}\}=q_{01}+q_{12}=c_{-1},\\
            &k=0:\tr\{\mathbf{\Lambda}_{0}\mathbf{Q}\}=q_{00} + q_{11} + q_{22}=c_0,
		\end{split}
	\end{equation}
	and so forth. The constant matrices $\mathbf{\Lambda}_k$ for degree-$2$ trigonometric polynomials are
 \begin{equation}
       \mathbf{\Lambda}_{-2}=\begin{pmatrix}
             0&0&0\\
             0&0&0\\
             1&0&0
        \end{pmatrix}~
        \mathbf{\Lambda}_{-1}=\begin{pmatrix}
             0&0&0\\
             1&0&0\\
             0&1&0
        \end{pmatrix}~
        \mathbf{\Lambda}_{0}=\begin{pmatrix}
             1&0&0\\
             0&1&0\\
             0&0&1
        \end{pmatrix}
 \end{equation}
 and so on.
\end{example}

Finally, we state the main theorem of this section.
\begin{theorem}[Semidefinite certificate of nonnegativity, {\citep[Theorem 2.5]{dumitrescu_positive_2017}}]\label{thrm:pos_sdp}
	A trigonometric polynomial is SOS \textit{iff} there exists a Gram matrix $\mathbf{Q}$ that is positive semidefinite. 
\end{theorem}
For reader's information, we notice that the smallest number $r$ of complex polynomials $h_j$ such that $f(z)=\sum_j h_j(z)h_j(z)^*$ is upper bounded by the rank of PSD matrices $Q$ satisfying linear constraints \eqref{appendix:sos_def}. Since $\mathbf{Q}$ has size $(N+1)\times (N+1)$, $r$ is at most $N+1$. To find an SOS representation with minimal number of $h_j$ is equivalent to a rank minimization problem.

Suppose that $f$ is SOS and that we have found $\mathbf{Q}\succeq 0$ being a Gram matrix of $f$. By the spectral decomposition of $\mathbf{Q}=\sum_j \sigma_j \mathbf{v}_j\mathbf{v}_j^*$ where $\sigma_j$ are the positive eigenvalues of $\mathbf{Q}$ and $\mathbf{v}_j$ the corresponding unitary family of eigenvectors, the following equation allows to find complex polynomials $h_j$:
\begin{align}
    f(z)&=\bm{\psi}(z)^*\mathbf{Q}\bm{\psi}(z)\\
    &=\sum_j \sigma_j (\bm{\psi}(z)^*\mathbf{v}_j)(\bm{\psi}(z)^*\mathbf{v}_j)^*
\end{align}
Thus $h_j(z)=\sqrt{\sigma_j}\bm{\psi}(z)^*\mathbf{v}_j$.

\subsection{Semidefinite representability of peak power functional}\label{sec:sdr_peak_power}
In this subsection, we show that the peak power is SDr, using the second type of semidefinite representation in Definition \ref{def:sdprep_2}, with respect to the Fourier coefficients of the nodal velocities $\mathbf{v}$. We recall that the peak power is a function defined as
\begin{equation}
	p[\mathbf{c(v)}]=\underset{t\in[0,\frac{2\pi}{\omega}]}{\max}\left\{\left\vert\sum_{n,m}\mathbf{c}_n(\mathbf{f})^T\mathbf{c}_m(\mathbf{v}) e^{i(n+m)\omega t}\right\vert\right\}.
\end{equation}
By collecting the coefficients of the complex exponential, we observe that
\begin{equation}
    \begin{split}
        &\sum_{n,m}\mathbf{c}_n(\mathbf{f})^T\mathbf{c}_m(\mathbf{v}) e^{i(n+m)\omega t}\\
        =&\sum_{k=-2N}^{2N}\sum_{n=-N}^{N}\mathbf{c}_{k-n}(\mathbf{f})^T\mathbf{c}_{n}(\mathbf{v})e^{ik\omega t},
    \end{split}
\end{equation}
with the inner summation conveniently written over $-N$ to $N$. We have extended the sequence $\{\mathbf{c}_n(\mathbf{f})\}$ for $|n|$ strictly larger than $N$ by $0$. In addition, set $z=e^{-i\omega t}$ and define $q_k = \sum_{n=-N}^{N}\mathbf{c}_{k-n}(\mathbf{f})^T\mathbf{c}_{n}(\mathbf{v})$. Then, it follows that
\begin{equation}
	p[\mathbf{c}(\mathbf{v})]=\underset{z\in\mathbb{T}}{\max}\left\vert q(z) \right\vert,
\end{equation}
where $q(z)=\sum_{k=-2N}^{2N}q_kz^{-k}$. We have $q_0=0$, since the average power over one period delivered by the harmonic load is equal to $0$.

Next, we assume $(\theta,\mathbf{c}(\mathbf{v}))$ in the epigraph of $p[.]$. Then, we have
\begin{equation}
    \begin{split}
        &\theta\geq p[\mathbf{c}(\mathbf{v})]\\
        \iff& \theta\pm q \text{ are positive trigonometric polynomials}.
    \end{split}
\end{equation}
By Theorem \ref{thrm:pos_sdp} and using the notation in Example~\ref{example:gram}, and since the coefficients of $\theta\pm q$ depend linearly on $\theta$ and $\mathbf{c}(\mathbf{v})$, $(\theta,\mathbf{c}(\mathbf{v}))$ is in the epigraph of the peak power function \textit{iff} there exists $\mathbf{Q}_1\in\mathbb{S}^{2N+1}_+$ and $\mathbf{Q}_2\in\mathbb{S}^{2N+1}_+$ that satisfy the equality constraints of \eqref{eqn:sdr_second_type}. Therefore, the peak power function is SDr. More precisely, $\theta\pm q$ are positive \textit{iff} the following two sequences of linear equations hold:
\begin{subequations}\label{eqn:sdr_peak_power}
    \begin{equation}
	\exists \mathbf{Q}_1\in\mathbb{S}^{2N+1}_+, \begin{cases}
	    \theta = \tr\{\mathbf{\Lambda}_0\mathbf{Q}_1\},\\
        q_k = \tr\{\mathbf{\Lambda}_k\mathbf{Q}_1\},\forall k\neq 0,
	\end{cases}
    \end{equation}
    \begin{equation}
    		\exists \mathbf{Q}_2\in\mathbb{S}^{2N+1}_+,\begin{cases}
	    -\theta = \tr\{\mathbf{\Lambda}_0\mathbf{Q}_2\},\\
        -q_k = \tr\{\mathbf{\Lambda}_k\mathbf{Q}_2\},\forall k\neq 0.
	   \end{cases}
    \end{equation}
\end{subequations}
Finally, we insert the expression $q_k=\sum_{n=-N}^{N} \mathbf{c}_{k-n}(\mathbf{f})^T\mathbf{c}_n(\mathbf{v})$.
\begin{example}[Single harmonic $N=1$]
	We take the Hermitian matrices $\mathbf{Q}_1,\mathbf{Q}_2\in\mathbb{S}^3$ and $\mathbf{f}(t)=\mathbf{c}_{-1}(\mathbf{f})e^{-i\omega t}+\mathbf{c}_1(\mathbf{f})e^{i\omega t}$:
	\begin{equation}
		\begin{split}
			q_{-2}&=\mathbf{c}_{-1}(\mathbf{f})^T\mathbf{c}_{-1}(\mathbf{u}),\\
			q_{-1}&=\mathbf{c}_{-1}(\mathbf{f})^T\mathbf{c}_{0}(\mathbf{u})+\mathbf{c}_{0}(\mathbf{f})^T\mathbf{c}_{-1}(\mathbf{u})=0,\\
			q_{1}&=\mathbf{c}_{1}(\mathbf{f})^T\mathbf{c}_{0}(\mathbf{u})+\mathbf{c}_{0}(\mathbf{f})^T\mathbf{c}_{1}(\mathbf{u})=0,\\
			q_{2}&=\mathbf{c}_{1}(\mathbf{f})^T\mathbf{c}_{1}(\mathbf{u}).
		\end{split}
	\end{equation}
\end{example}
\begin{example}[Double harmonic $N=2$]
	We take $\mathbf{Q}_1,\mathbf{Q}_2\in\mathbb{S}^5$ and $$f(t)=\mathbf{c}_{-2}(\mathbf{f})e^{-2i\omega t}+\mathbf{c}_{-1}(\mathbf{f})e^{-i\omega t}+\mathbf{c}_{1}(\mathbf{f})e^{i\omega t}+\mathbf{c}_{2}(\mathbf{f})e^{2i\omega t}.$$
	Therefore, it holds that
	\begin{equation}
		\begin{split}
			q_{-4}&=\mathbf{c}_{-2}(\mathbf{f})^T\mathbf{c}_{-2}(\mathbf{u})+\mathbf{c}_{-3}(\mathbf{f})^T\mathbf{c}_{-1}(\mathbf{u})+\mathbf{c}_{-4}(\mathbf{f})^T\mathbf{c}_{0}(\mathbf{u})\\
                &+\mathbf{c}_{-5}(\mathbf{f})^T\mathbf{c}_{1}(\mathbf{u})+\mathbf{c}_{-6}(\mathbf{f})^T\mathbf{c}_{2}(\mathbf{u})\\
				&=\mathbf{c}_{-2}(\mathbf{f})^T\mathbf{c}_{-2}(\mathbf{u}),\\
			q_{-3}&=\mathbf{c}_{-1}(\mathbf{f})^T\mathbf{c}_{-2}(\mathbf{u})+\mathbf{c}_{-2}(\mathbf{f})^T\mathbf{c}_{-1}(\mathbf{u})+\mathbf{c}_{-3}(\mathbf{f})^T\mathbf{c}_{0}(\mathbf{u})\\
                &+\mathbf{c}_{-4}(\mathbf{f})^T\mathbf{c}_{1}(\mathbf{u})+\mathbf{c}_{-5}(\mathbf{f})^T\mathbf{c}_{2}(\mathbf{u})\\
				&=\mathbf{c}_{-1}(\mathbf{f})^T\mathbf{c}_{-2}(\mathbf{u})+\mathbf{c}_{-2}(\mathbf{f})^T\mathbf{c}_{-1}(\mathbf{u}),\\
			q_{-2}&=\mathbf{c}_{0}(\mathbf{f})^T\mathbf{c}_{-2}(\mathbf{u})+\mathbf{c}_{-1}(\mathbf{f})^T\mathbf{c}_{-1}(\mathbf{u})+\mathbf{c}_{-2}(\mathbf{f})^T\mathbf{c}_{0}(\mathbf{u})\\&+\mathbf{c}_{-3}(\mathbf{f})^T\mathbf{c}_{1}(\mathbf{u})+\mathbf{c}_{-4}(\mathbf{f})^T\mathbf{c}_{2}(\mathbf{u})\\
				&=\mathbf{c}_{-1}(\mathbf{f})^T\mathbf{c}_{-1}(\mathbf{u}),\\
			q_{-1}&=\mathbf{c}_{1}(\mathbf{f})^T\mathbf{c}_{-2}(\mathbf{u})+\mathbf{c}_{0}(\mathbf{f})^T\mathbf{c}_{-1}(\mathbf{u})+\mathbf{c}_{-1}(\mathbf{f})^T\mathbf{c}_{0}(\mathbf{u})\\&+\mathbf{c}_{-2}(\mathbf{f})^T\mathbf{c}_{1}(\mathbf{u})+\mathbf{c}_{-3}(\mathbf{f})^T\mathbf{c}_{2}(\mathbf{u})\\
				&=\mathbf{c}_{1}(\mathbf{f})^T\mathbf{c}_{-2}(\mathbf{u})+\mathbf{c}_{-2}(\mathbf{f})^T\mathbf{c}_1(\mathbf{u}).\\
		\end{split}
	\end{equation}
	Then, by the symmetry condition, we receive
		\begin{equation}
		\begin{split}
			q_{4}&=\mathbf{c}_{2}(\mathbf{f})^T\mathbf{c}_{2}(\mathbf{u}),\\
			q_{3}&=\mathbf{c}_{1}(\mathbf{f})^T\mathbf{c}_{2}(\mathbf{u})+\mathbf{c}_{2}(\mathbf{f})^T\mathbf{c}_{1}(\mathbf{u}),\\
			q_{2}&=\mathbf{c}_{1}(\mathbf{f})^T\mathbf{c}_{1}(\mathbf{u}),\\
			q_{1}&=\mathbf{c}_{-1}(\mathbf{f})^T\mathbf{c}_{2}(\mathbf{u})+\mathbf{c}_{2}(\mathbf{f})^T\mathbf{c}_{-1}(\mathbf{u}).\\
		\end{split}
	\end{equation}	
\end{example}
\section{Independence of peak power of non-physical solution}\label{appendix:B}
To state the independence of peak power with non-physical solution, we need to introduce the generalized eigenvalue problem related to $\mathbf{K(a)}$ and $\mathbf{M(a)}$. The generalized eigenvalue problem consists of finding $\lambda$ and $w$ such that
\begin{equation}\label{eqn:eigenproblem}
    \mathbf{K(a)w}=\lambda \mathbf{\mathbf{M(a)}\mathbf{w}}.
\end{equation}
By the construction of $\mathbf{M(a)}$ and $\mathbf{K(a)}$, we can show that $\ker{\mathbf{M(a)}}\subseteq\ker{\mathbf{K(a)}}$ (see \citep{achtziger_structural_2008}). Any solution $w$ in the null space of $\mathbf{M(a)}$ is thus a trivial solution of the eigenproblem. To eliminate this case, we define the so-called well-defined eigenproblem:
\begin{definition}[Well-defined eigenproblem]
    The well-defined solution of \eqref{eqn:eigenproblem} is a pair of $(\lambda,w)$ such that
    \begin{equation}
        \mathbf{K(a)}w=\lambda \mathbf{M(a)}\mathbf{w}
    \end{equation}
    and $\mathbf{w}\notin\ker{\mathbf{M(a)}}$.
\end{definition}
For any feasible $\mathbf{a}$, we can show that there is a minimal solution $\lambda_{\min}(\mathbf{a})>0$ of the well-defined eigenproblem. The following proposition is true for $\lambda_{\min}(\mathbf{a})$:
\begin{proposition}[LMI characterization of minimal solution, {\citep[Proposition 2.2(c)]{achtziger_structural_2008}}]
    \begin{equation}
        \forall\omega, \omega^2<\lambda_{\min}(\mathbf{a})\iff -\omega^2\mathbf{M(a)}+\mathbf{K(a)}\succeq 0.
    \end{equation}
\end{proposition}

\begin{lemma}[Characterization of the range of dynamic stiffness matrices]\label{lemma:same_range}
    If $(\mathbf{a},\mathbf{c(v)})$ is feasible for the peak power minimization, then all the dynamic stiffness matrices $\mathbf{K}_{n\omega}(\mathbf{a})=-n^2\omega^2\mathbf{M(a)}+\mathbf{K(a)}$ have the same range.
\end{lemma}
\begin{proof}
$(\mathbf{a},\mathbf{c(v)})$ is feasible iff the following constraints are satisfied:
\begin{align}
    a_i\geq0, \mathbf{a}^T\mathbf{q}\leq m, -N^2\omega^2\mathbf{M(a)}+\mathbf{K(a)}\succeq 0,\\
    (-k^2\omega^2\mathbf{M(a)}+\mathbf{K(a)})\mathbf{c}_k(v)=ik\omega \mathbf{c}_k(\mathbf{f}),\forall k.
\end{align}

By the LMI constraint, each of the $n^2\omega^2$ is strictly less than $\lambda_{\min}(\mathbf{a})$. For a fixed $n$, let us consider the solution of \begin{equation}
    \mathbf{K}_{n\omega}(\mathbf{a})\mathbf{c}_n(\mathbf{v})=in\omega \mathbf{c}_n(\mathbf{f}),
\end{equation}
It must have the form $\mathbf{c}_n(v)=in\mathbf{K}_{n\omega}(\mathbf{a})^\dagger \mathbf{c}_n(\mathbf{f})+\mathbf{c}_0$ for any $\mathbf{c}_0\in\ker{\mathbf{K}_{n\omega}(\mathbf{a})}$. If $\mathbf{c}_0\notin \ker{\mathbf{M(\mathbf{a})}}$, then $n^2\omega^2$ would be a well-defined solution of the generalized eigenvalue problem. However, this is not possible since $n^2\omega^2\leq N^2\omega^2<\lambda_{\min}(\mathbf{a})$.

What we have shown is, in fact, that $\ker{\mathbf{K}_{n\omega}}(\mathbf{a})=\ker{\mathbf{M(\mathbf{a})}}$ for all $n$. By orthogonal complementarity, the range of $\ker{\mathbf{K}_{n\omega}}(\mathbf{a})$ is the same for all $n$.    
\end{proof}

Finally, if all $\mathbf{c}_n(\mathbf{f})$ are in the range of $\mathbf{K}_{n\omega}(\mathbf{a})$, 
\begin{equation}
    \forall m,\mathbf{c}_m(\mathbf{f})^T\mathbf{c}_n(v)=\mathbf{c}_m(\mathbf{f})^T\mathbf{K}_{n\omega}(\mathbf{a})^\dagger \mathbf{c}_n(\mathbf{f}).
\end{equation}
The peak power is thus independent of the nonphysical solution, since its semidefinite representation is also independent of the nonphysical solution.

\section{Convex relaxation of the peak power minimization}\label{appendix:peak_power_minimization_multi}

We recall the peak power minimization under equilibrium equation constraint
\begin{equation}\label{eqn:pp_appendix_C}\tag{$\mathcal{P}_{\text{pp}}$}
    \begin{split}
        \underset{\mathbf{a},\mathbf{c(v)},\theta,\mathbf{Q}_1,\mathbf{Q}_2}{\min} ~&~ \theta\\
        \text{ s.t } & \left\{\begin{array}{l}
            a_i\geq0, \mathbf{a}^T\mathbf{q}\leq m,\\
            \mathbf{K}_{N\omega}(\mathbf{a})\succeq 0,\\
            \mathbf{K}_{k\omega}(\mathbf{a})\mathbf{c}_k(\mathbf{v})=ik\omega \mathbf{c}_k(\mathbf{f}),\forall k,\\
            \theta = \tr\{\mathbf{\Lambda}_0\mathbf{Q}_1\}=-\tr\{\mathbf{\Lambda}_0\mathbf{Q}_2\},\\
            q_k = \tr\{\mathbf{\Lambda}_k\mathbf{Q}_1\}=-\tr\{\mathbf{\Lambda}_k\mathbf{Q}_2\},\forall k\neq 0,\\
            \mathbf{Q}_1,\mathbf{Q}_2\succeq 0.
        \end{array}\right.
    \end{split}
\end{equation}
The coefficients $q_k$ are written as
\begin{equation}
    \forall k\in\{-2N,\dots,2N\},q_k=\sum_{n=-N}^{N}\mathbf{c}_{k-n}(\mathbf{f})^T\mathbf{c}_{n}(\mathbf{v}),
\end{equation}
with the convention that $\mathbf{c}_m(\mathbf{f})=0$ if $|m|>N$. Due to the symmetry condition $\mathbf{c}_{-k}(\mathbf{v})=\overline{\mathbf{c}_k(\mathbf{v})}$, we only need to consider equilibrium equations $\mathbf{K}_{k\omega}(\mathbf{a})\mathbf{c}_k(\mathbf{v})=ik\omega \mathbf{c}_k(\mathbf{f})$ for $k\in\{1,\dots,N\}$. Thus,
\begin{equation}
    \begin{split}
        q_k&=\sum_{n=-N}^{N}\mathbf{c}_{k-n}(\mathbf{f})^T\mathbf{c}_{n}(\mathbf{v})\\
		&=\sum_{n=1}^{N}\mathbf{c}_{n-k}(\mathbf{f})^*\mathbf{c}_{n}(\mathbf{v})+\mathbf{c}_{n}(\mathbf{v})^*\mathbf{c}_{k+n}(\mathbf{f})\\
		&=\sum_{n=1}^{N}\mathbf{c}_{n-k}(\mathbf{f})^*K_{n\omega}(\mathbf{a})^\dagger in\omega \mathbf{c}_{n}(\mathbf{f})\\&+\sum_{n=1}^{N}(in\omega \mathbf{c}_{n}(\mathbf{f}))^*K_{n\omega}(\mathbf{a})^\dagger \mathbf{c}_{k+n}(\mathbf{f}).
    \end{split}
\end{equation}
Let us define $D\mathbf{c}(\mathbf{f})=\begin{pmatrix}
    i\omega \mathbf{c}_1(\mathbf{f})^T & \dots & iN\omega \mathbf{c}_N(\mathbf{f})^T
\end{pmatrix}^T$, a column vector with block structure such that each block corresponds to the vector $in\omega \mathbf{c}_n(\mathbf{f})$. Similarly, we consider $\mathbf{c}(\mathbf{f})$, a column block vector with block $\mathbf{c}_n(\mathbf{f})$. Using the shifting operator $T_k\mathbf{c}(\mathbf{f})$ such that the $n$-th block of $T_k\mathbf{c}(\mathbf{f})$ is $\mathbf{c}_{n+k}(\mathbf{f})$, we can see $q_k$ as
\begin{equation}
    \begin{split}
        q_k &= (T_{-k}\mathbf{c}(\mathbf{f}))^*\mathbf{L}_{N,\omega}(\mathbf{a})^\dagger D\mathbf{c}(\mathbf{f})\\
        &+(D\mathbf{c}(\mathbf{f}))^*\mathbf{L}_{N,\omega}(\mathbf{a})^\dagger T_{k}\mathbf{c}(\mathbf{f}),
    \end{split}
\end{equation}
where $\mathbf{L}_{N,\omega}(\mathbf{a})=\diag\{\mathbf{K}_\omega(\mathbf{a}),\dots,\mathbf{K}_{N\omega}(\mathbf{a})\}$. By convention, $\mathbf{c}_m(\mathbf{f})$ is zero for $m>N$. We see that $T_kc(\mathbf{f})=0$ for any $k\geq N$. Thus, we define the matrix $\mathbf{F}$ with $3N$ columns such that the $k$-th column is
\begin{equation}
    \forall k\in\{1,\dots,3N\}, \mathbf{F}_k= \begin{cases}
        T_{N-k}\mathbf{c}(\mathbf{f})& \text{ if $k\neq N$}\\
        D\mathbf{c}(\mathbf{f})& \text{ if $k=N$}
    \end{cases}
\end{equation}

For example, consider $N=3$ and $c(\mathbf{f})=\begin{pmatrix}
    \mathbf{c}_1(\mathbf{f})^T & \mathbf{c}_2(\mathbf{f})^T & \mathbf{c}_3(\mathbf{f})^T
\end{pmatrix}^T$. The matrix $\mathbf{F}$ is then
\begin{equation}
   \mathbf{F}=\begin{pmatrix}
    \mathbf{c}_3(\mathbf{f})&\mathbf{c}_2(\mathbf{f})&i\omega \mathbf{c}_1(\mathbf{f})&\mathbf{c}_0(\mathbf{f})&\overline{\mathbf{c}_1(\mathbf{f})}&\dots&0\\
    0&\mathbf{c}_3(\mathbf{f})&2i\omega \mathbf{c}_2(\mathbf{f})&\mathbf{c}_1(\mathbf{f})&\mathbf{c}_0(\mathbf{f})&\dots&0\\
    0&0&3i\omega \mathbf{c}_3(\mathbf{f})&\mathbf{c}_2(\mathbf{f})&\mathbf{c}_1(\mathbf{f})&\dots&\overline{\mathbf{c}_3(\mathbf{f})}
    \end{pmatrix}.
\end{equation}
If $\mathbf{X}=\mathbf{F}^*\mathbf{L}_{N,\omega}(\mathbf{a})^\dagger \mathbf{F}$, then there are constant matrices $\mathbf{C}_k$ such that
\begin{equation}
    q_k = \tr\{\mathbf{C}_k\mathbf{X}\}.
\end{equation}
By introducing the Hermitian variable $\mathbf{X}$ of size $3N\times 3N$, the peak power minimization problem \eqref{eqn:pp_appendix_C} is equivalent to
\begin{equation}
    \begin{split}
        \underset{\mathbf{a},\mathbf{X},\theta,\mathbf{Q}_1,\mathbf{Q}_2}{\min} ~&~ \theta\\
        \text{ s.t } & \left\{\begin{array}{l}
            a_i\geq0, \mathbf{a}^T\mathbf{q}\leq m,\\
            \mathbf{K}_{N\omega}(\mathbf{a})\succeq 0,\\
            \mathbf{c}_k(\mathbf{f})\in\range\{\mathbf{K}_{k\omega}(\mathbf{a})\},\forall k,\\
            \mathbf{X}=\mathbf{F}^*\mathbf{L}_{N,\omega}(\mathbf{a})^\dagger \mathbf{F},\\
            \theta = \tr\{\mathbf{\Lambda}_0\mathbf{Q}_1\}=-\tr\{\mathbf{\Lambda}_0\mathbf{Q}_2\},\\
            \tr\{\mathbf{C}_k\mathbf{X}\} = \tr\{\mathbf{\Lambda}_k\mathbf{Q}_1\},\forall k\neq 0,\\
            \tr\{\mathbf{C}_k\mathbf{X}\} =-\tr\{\mathbf{\Lambda}_k\mathbf{Q}_2\},\forall k\neq 0,\\
            \mathbf{Q}_1,\mathbf{Q}_2\succeq 0,\\
        \end{array}\right.
    \end{split}.
\end{equation}
Finally, using the techniques presented in Sections \ref{sec:conv_relax} and \ref{sec:lag_relax}, the previous problem is at first equivalent to
\begin{equation}
    \begin{split}
        \underset{\mathbf{a},\mathbf{X},\theta,\mathbf{Q}_1,\mathbf{Q}_2}{\min} ~&~ \theta\\
        \text{ s.t } & \left\{\begin{array}{l}
             a_i\geq0, \mathbf{a}^T\mathbf{q}\leq m\\
            \begin{pmatrix}
                \mathbf{X} & \mathbf{F}^* \\
                \mathbf{F} & \mathbf{L}_{N,\omega}(\mathbf{a})
            \end{pmatrix}\succeq 0,\\
            \mathbf{X}=\mathbf{F}^*\mathbf{L}_{N,\omega}(\mathbf{a})^\dagger \mathbf{F},\\
            \theta = \tr\{\mathbf{\Lambda}_0\mathbf{Q}_1\}=-\tr\{\mathbf{\Lambda}_0\mathbf{Q}_2\},\\
            \tr\{\mathbf{C}_k\mathbf{X}\} = \tr\{\mathbf{\Lambda}_k\mathbf{Q}_1\},\forall k\neq 0,\\
            \tr\{\mathbf{C}_k\mathbf{X}\} =-\tr\{\mathbf{\Lambda}_k\mathbf{Q}_2\},\forall k\neq 0,\\
            \mathbf{Q}_1,\mathbf{Q}_2\succeq 0,\\
        \end{array}\right.
    \end{split}
\end{equation}
Note that the constraint $\mathbf{K}_{N\omega}(\mathbf{a})\succeq 0$ has disappeared from the optimization problem however it doesn't change the constraint set of the variables since the constraint $\begin{pmatrix}
    \mathbf{X} & \mathbf{F}^* \\
    \mathbf{F} & \mathbf{L}_{N,\omega}(\mathbf{a})
\end{pmatrix}\succeq 0$ would imply that $\mathbf{K}_{N\omega}(\mathbf{a})\succeq 0$ by the block structure. As a result, the Lagrange relaxation reads as
\begin{equation}
    \begin{split}
        \underset{\mathbf{a},\mathbf{X},\theta,\mathbf{Q}_1,\mathbf{Q}_2}{\min} ~&~ \theta+\eta\tr\{\mathbf{X}-\mathbf{F}^*\mathbf{L}_{N,\omega}(\mathbf{a})^\dagger \mathbf{F}\}\\
        \text{ s.t } & \left\{\begin{array}{l}
             a_i\geq0, \mathbf{a}^T\mathbf{q}\leq m\\
            \begin{pmatrix}
                \mathbf{X} & \mathbf{F}^* \\
                \mathbf{F} & \mathbf{L}_{N,\omega}(\mathbf{a})
            \end{pmatrix}\succeq 0,\\
            \theta = \tr\{\mathbf{\Lambda}_0\mathbf{Q}_1\}=-\tr\{\mathbf{\Lambda}_0\mathbf{Q}_2\},\\
            \tr\{\mathbf{C}_k\mathbf{X}\} = \tr\{\mathbf{\Lambda}_k\mathbf{Q}_1\},\forall k\neq 0,\\
            \tr\{\mathbf{C}_k\mathbf{X}\} =-\tr\{\mathbf{\Lambda}_k\mathbf{Q}_2\},\forall k\neq 0,\\
            \mathbf{Q}_1,\mathbf{Q}_2\succeq 0,\\
        \end{array}\right.
    \end{split}.
\end{equation}
By discarding the non convex term $\tr\{\mathbf{F}^*\mathbf{L}_{N,\omega}(\mathbf{a})^\dagger \mathbf{F}\}$ we find the convex relaxation with penalty which evaluates as
\begin{equation}\label{eqn:multiple_harmonics}
    \begin{split}
        \underset{\mathbf{a},\mathbf{X},\theta,\mathbf{Q}_1,\mathbf{Q}_2}{\min} ~&~ \theta+\eta\tr\{\mathbf{X}\}\\
        \text{ s.t } & \left\{\begin{array}{l}
             a_i\geq0, \mathbf{a}^T\mathbf{q}\leq m,\\
            \begin{pmatrix}
                \mathbf{X} & \mathbf{F}^* \\
                \mathbf{F} & \mathbf{L}_{N,\omega}(\mathbf{a})
            \end{pmatrix}\succeq 0,\\
            \theta = \tr\{\mathbf{\Lambda}_0\mathbf{Q}_1\}=-\tr\{\mathbf{\Lambda}_0\mathbf{Q}_2\},\\
            \tr\{\mathbf{C}_k\mathbf{X}\} = \tr\{\mathbf{\Lambda}_k\mathbf{Q}_1\},\forall k\neq 0\\
            \tr\{\mathbf{C}_k\mathbf{X}\}=-\tr\{\mathbf{\Lambda}_k\mathbf{Q}_2\},\forall k\neq 0,\\
            \mathbf{Q}_1,\mathbf{Q}_2\succeq 0,\\
        \end{array}\right.
    \end{split}.
\end{equation}

\section{Sensitivity analysis of the peak power}\label{appendix:d}
We propose here a way to compute the sensitivity of the peak power function with respect to the design variables $a$ using the adjoint model. It will thus be possible to use a general non-linear programming method to find local minima of the peak power, but this is not considered in this study.

To compute the sensitivity of a function $\phi[\mathbf{c(v)}]$ with respect to $a$ such that $\mathbf{K}_{n\omega}\mathbf{c}_n(\mathbf{v})=in\omega \mathbf{c}_n(\mathbf{f})$ for $n\in\{-N,\dots,N\}$, we need to write the augmented function $\tilde{\phi}[\mathbf{c(v)}]$ with $\bm{\lambda}_n = \overline{\bm{\lambda}_{-n}}$:
\begin{equation}
    \begin{split}
        \tilde{\phi}[\mathbf{c(v)}]&=\phi[\mathbf{c(v)}]\\
        &+\sum_{n=-N}^N\bm{\lambda}_n^T(\mathbf{K}_{n\omega}(\mathbf{a})\mathbf{c}_n(\mathbf{v})-in\omega \mathbf{c}_n(\mathbf{f})).
    \end{split}
\end{equation}
For the sensitivity with respect to $a_j$, we differentiate and obtain
\begin{align}
    \frac{\partial\tilde{\phi}}{\partial a_j}&=\sum_{n=-N}^{N}\nabla_{\mathbf{c}_n(\mathbf{v})}\phi[\mathbf{c(v)}]^T\frac{\partial \mathbf{c}_n(\mathbf{v})}{\partial a_j}\\
    &+\sum_{n=-N}^N\bm{\lambda}_n^T\frac{\partial \mathbf{K}_{n\omega}(\mathbf{a})}{\partial a_j}\mathbf{c}_n(\mathbf{v})+\bm{\lambda}_n^T\mathbf{K}_{n\omega}(\mathbf{a})\frac{\partial \mathbf{c}_n(\mathbf{v})}{\partial a_j}.
\end{align}
If $\bm{\lambda}_n$ are the solutions of the adjoint models $\mathbf{K}_{n\omega}(\mathbf{a})\bm{\lambda}_n=-\nabla_{\mathbf{c}_n(v)}\phi[\mathbf{c(v)}]$ for $n\in\{-N,\dots,N\}$, then 
$$\forall j, \frac{\partial{\phi}}{\partial a_j}=\frac{\partial\tilde{\phi}}{\partial a_j}=\sum_{n=-N}^N\bm{\lambda}_n^T\frac{\partial \mathbf{K}_{n\omega}(\mathbf{a})}{\partial a_j}\mathbf{c}_n(\mathbf{v}).$$

For the peak power function $p[\mathbf{c(v)}]$ in particular, it is the optimal value of the linear SDP problem
\begin{equation}\label{eqn:peak_power_sdp}
    \begin{array}{cc}
        \underset{\theta,\mathbf{Q}_0,\mathbf{Q}_1}{\min} & \theta  \\
         \text{s.t. } & \begin{cases}
             \theta = \tr\{\bm{\Lambda}_0\mathbf{Q}_1\}=-\tr\{\bm{\Lambda}_0\mathbf{Q}_2\},\\
            q_k = \tr\{\bm{\Lambda}_k\mathbf{Q}_1\}=-\tr\{\bm{\Lambda}_k\mathbf{Q}_2\},\forall k\neq 0,\\
            \mathbf{Q}_1,\mathbf{Q}_2\succeq 0,
         \end{cases} 
    \end{array}
\end{equation}
where we insert the expression $q_k=\sum_{n=-N}^{N}\mathbf{c}_{k-n}(\mathbf{f})^T\mathbf{c}_{n}(\mathbf{v})$, which is linear in $\mathbf{c(v)}$. To compute the sensitivity of the peak power, we need to compute the sensitivity of a linear SDP program with respect to its data defining the linear constraints. This problem has recently been investigated in \citep{agrawal2020differentiating}, which resulted in implementation of the technique in the Python package \textbf{CVXPY} \citep{diamond2016cvxpy}.

Let $\mathcal{P}(\mathbf{q})$ denotes the optimal value of \eqref{eqn:peak_power_sdp} as a function of $\mathbf{q}$. by the chain rule we obtain that
\begin{equation}
    \nabla_{\mathbf{c}_n(v)}p[\mathbf{c(v)}]=\sum_{k=-2N}^{2N}\frac{\partial\mathcal{P}}{\partial q_k}\mathbf{c}_{k-n}(\mathbf{f}).
\end{equation}
Thus, we can compute the sensitivity $\frac{\partial p[\mathbf{c(v)}]}{\partial a_j}$ using the adjoint model method. At each fixed design vector $a$, we first solve for direct models $\mathbf{K}_{n\omega}(\mathbf{a})\mathbf{c}_n(v)=in\omega \mathbf{c}_n(\mathbf{f})$, followed by the solution of adjoint models $\forall n\in\{-N,\dots,N\}$:
\begin{equation}\label{eqn:adjoint_model}
    \mathbf{K}_{n\omega}(\mathbf{a})\bm{\lambda}_n = \sum_{k=-2N}^{2N}\frac{\partial\mathcal{P}}{\partial q_k}\mathbf{c}_{k-n}(\mathbf{f}).
\end{equation}
Then, the sensitivities are evaluated as $\frac{\partial p[\mathbf{c(v)}]}{\partial a_j}=\sum_{n=-N}^{N}\bm{\lambda}_n^T\frac{\partial\mathbf{K}_{n\omega}(\mathbf{a})}{\partial a_j}\mathbf{c}_n(v)$.

If $\mathbf{a}$ satisfies the LMI $\begin{pmatrix}
    \mathbf{X} & \mathbf{F}\\
    \mathbf{F}^*& \mathbf{L}_{N,\omega}(\mathbf{a})
\end{pmatrix}\succeq 0$, it is guaranteed that the right hand side of the adjoint model \eqref{eqn:adjoint_model} is in the range of $\mathbf{K}_{n\omega}(\mathbf{a})$.

\end{appendices}

\section*{Statements and Declarations}
\subsection*{Funding}
This work is funded by Czech Science Foundation (grant number 22-15524S).
\subsection*{Competing interests}
The authors declare that they have no conflict of interest.

\bibliographystyle{abbrvnat}
\bibliography{reference}

\begin{thebibliography}{32}
\providecommand{\natexlab}[1]{#1}
\providecommand{\url}[1]{\texttt{#1}}
\expandafter\ifx\csname urlstyle\endcsname\relax
  \providecommand{\doi}[1]{doi: #1}\else
  \providecommand{\doi}{doi: \begingroup \urlstyle{rm}\Url}\fi

\bibitem[Achtziger and Kočvara(2008)]{achtziger_structural_2008}
W.~Achtziger and M.~Kočvara.
\newblock Structural topology optimization with eigenvalues.
\newblock \emph{SIAM J. Optim.}, 18\penalty0 (4):\penalty0 1129--1164, 2008.
\newblock ISSN 1052-6234.
\newblock \doi{10.1137/060651446}.

\bibitem[Achtziger et~al.(1992)Achtziger, Bends{\o}e, Ben-Tal, and
  Zowe]{Achtziger1992}
W.~Achtziger, M.~P. Bends{\o}e, A.~Ben-Tal, and J.~Zowe.
\newblock Equivalent displacement based formulations for maximum strength truss
  topology design.
\newblock \emph{Impact Comput. Sci. Eng.}, 4\penalty0 (4):\penalty0 315--345,
  1992.
\newblock \doi{10.1016/0899-8248(92)90005-s}.

\bibitem[Agrawal et~al.(2020)Agrawal, Barratt, Boyd, Busseti, and
  Moursi]{agrawal2020differentiating}
A.~Agrawal, S.~Barratt, S.~Boyd, E.~Busseti, and W.~M. Moursi.
\newblock Differentiating through a cone program, 2020.

\bibitem[ApS(2023)]{mosek}
M.~ApS.
\newblock \emph{MOSEK Optimizer API for Python}, 2023.
\newblock URL \url{http://docs.mosek.com/10.0/pythonapi/index.html}.

\bibitem[Aroztegui and Pessoa(2024)]{Aroztegui2024}
J.~M. Aroztegui and A.~Pessoa.
\newblock A cutting plane approach to maximization of fundamental frequency in
  truss topology optimization.
\newblock \emph{Struct. Multidiscipl. Optim.}, 67\penalty0 (4), 2024.
\newblock ISSN 1615-1488.
\newblock \doi{10.1007/s00158-024-03778-y}.

\bibitem[Ben-Tal and Bends\o{}e(1993)]{BentalBendsoe1993}
A.~Ben-Tal and M.~P. Bends\o{}e.
\newblock A new method for optimal truss topology design.
\newblock \emph{SIAM Journal on Optimization}, 3\penalty0 (2):\penalty0
  322--358, 1993.
\newblock \doi{10.1137/0803015}.
\newblock URL \url{https://doi.org/10.1137/0803015}.

\bibitem[Ben-Tal and Nemirovski(1997)]{BenTal1997}
A.~Ben-Tal and A.~Nemirovski.
\newblock Robust truss topology design via semidefinite programming.
\newblock \emph{{SIAM} J. Optim.}, 7\penalty0 (4):\penalty0 991--1016, 1997.
\newblock \doi{10.1137/s1052623495291951}.

\bibitem[Ben-Tal and Nemirovski(2001{\natexlab{a}})]{Ben_Tal_2001}
A.~Ben-Tal and A.~Nemirovski.
\newblock \emph{Lectures on modern convex optimization}.
\newblock Society for Industrial and Applied Mathematics, 2001{\natexlab{a}}.
\newblock \doi{10.1137/1.9780898718829}.

\bibitem[Ben-Tal and Nemirovski(2001{\natexlab{b}})]{ben-tal_lectures_2001}
A.~Ben-Tal and A.~Nemirovski.
\newblock \emph{Lectures on modern convex optimization}.
\newblock {MOS}-{SIAM} {Series} on {Optimization}. Society for Industrial and
  Applied Mathematics, 2001{\natexlab{b}}.
\newblock ISBN 978-0-89871-491-3.
\newblock \doi{10.1137/1.9780898718829}.

\bibitem[Boyd and Vandenberghe(2004)]{boyd_convex_2004}
S.~Boyd and L.~Vandenberghe.
\newblock \emph{Convex {Optimization}}.
\newblock Cambridge University Press, 2004.
\newblock ISBN 9780511804441.
\newblock \doi{10.1017/CBO9780511804441}.

\bibitem[Diamond and Boyd(2016)]{diamond2016cvxpy}
S.~Diamond and S.~Boyd.
\newblock {CVXPY}: {A} {P}ython-embedded modeling language for convex
  optimization.
\newblock \emph{J. Mach. Learn. Res.}, 17\penalty0 (83):\penalty0 1--5, 2016.

\bibitem[Dorn et~al.(1964)Dorn, Gomory, and Greenberg]{Dorn1964}
W.~S. Dorn, R.~E. Gomory, and H.~J. Greenberg.
\newblock {Automatic design of optimal structures}.
\newblock \emph{J. Mec.}, 3\penalty0 (1):\penalty0 25--52, 1964.

\bibitem[Dumitrescu(2017)]{dumitrescu_positive_2017}
B.~Dumitrescu.
\newblock \emph{Positive trigonometric polynomials and signal processing
  applications}.
\newblock Signals and {Communication} {Technology}. Springer International
  Publishing, Cham, 2017.
\newblock ISBN 978-3-319-53688-0.
\newblock \doi{10.1007/978-3-319-53688-0}.

\bibitem[Heidari et~al.(2009)Heidari, Cogill, Allaire, and
  Sheth]{heidari_optimization_2009}
M.~Heidari, R.~Cogill, P.~Allaire, and P.~Sheth.
\newblock Optimization of peak power in vibrating structures via semidefinite
  programming.
\newblock In \emph{50th {AIAA}/{ASME}/{ASCE}/{AHS}/{ASC} {Structures},
  {Structural} {Dynamics}, and {Materials} {Conference}}. American Institute of
  Aeronautics and Astronautics, 2009.
\newblock ISBN 978-1-60086-975-4.
\newblock \doi{10.2514/6.2009-2181}.

\bibitem[Kang et~al.(2006)Kang, Park, and Arora]{Kang2006}
B.-S. Kang, G.-J. Park, and J.~S. Arora.
\newblock A review of optimization of structures subjected to transient loads.
\newblock \emph{Struct. Multidiscipl. Optim.}, 31\penalty0 (2):\penalty0
  81--95, 2006.
\newblock \doi{10.1007/s00158-005-0575-4}.

\bibitem[Kim et~al.(2011)Kim, Kojima, Mevissen, and
  Yamashita]{kim_exploiting_2011}
S.~Kim, M.~Kojima, M.~Mevissen, and M.~Yamashita.
\newblock Exploiting sparsity in linear and nonlinear matrix inequalities via
  positive semidefinite matrix completion.
\newblock \emph{Math. Program.}, 129\penalty0 (1):\penalty0 33--68, 2011.
\newblock ISSN 1436-4646.
\newblock \doi{10.1007/s10107-010-0402-6}.

\bibitem[Kočvara(2021)]{kocvara_decomposition_2021}
M.~Kočvara.
\newblock Decomposition of arrow type positive semidefinite matrices with
  application to topology optimization.
\newblock \emph{Math. Program.}, 190\penalty0 (1):\penalty0 105--134, Nov.
  2021.
\newblock ISSN 1436-4646.
\newblock \doi{10.1007/s10107-020-01526-w}.

\bibitem[Liu et~al.(2015{\natexlab{a}})Liu, Zhang, and Gao]{Liu2015}
H.~Liu, W.~Zhang, and T.~Gao.
\newblock Structural topology optimization under rotating load.
\newblock \emph{Struct. Multidiscipl. Optim.}, 53\penalty0 (4):\penalty0
  847--859, 2015{\natexlab{a}}.
\newblock \doi{10.1007/s00158-015-1356-3}.

\bibitem[Liu et~al.(2015{\natexlab{b}})Liu, Zhang, and Gao]{Liu2015rev}
H.~Liu, W.~Zhang, and T.~Gao.
\newblock A comparative study of dynamic analysis methods for structural
  topology optimization under harmonic force excitations.
\newblock \emph{Struct. Multidiscipl. Optim.}, 51\penalty0 (6):\penalty0
  1321--1333, 2015{\natexlab{b}}.
\newblock \doi{10.1007/s00158-014-1218-4}.

\bibitem[Liu et~al.(2008)Liu, Yan, and Cheng]{Liu2008}
L.~Liu, J.~Yan, and G.~Cheng.
\newblock Optimum structure with homogeneous optimum truss-like material.
\newblock \emph{Comput. Struct.}, 86\penalty0 (13-14):\penalty0 1417--1425,
  2008.
\newblock \doi{10.1016/j.compstruc.2007.04.030}.

\bibitem[Lobo et~al.(1998)Lobo, Vandenberghe, Boyd, and Lebret]{Lobo_1998}
M.~S. Lobo, L.~Vandenberghe, S.~Boyd, and H.~Lebret.
\newblock Applications of second-order cone programming.
\newblock \emph{Linear Alg. Appl.}, 284\penalty0 (1-3):\penalty0 193--228,
  1998.
\newblock \doi{10.1016/s0024-3795(98)10032-0}.

\bibitem[Ma et~al.(1993)Ma, Kikuchi, and Hagiwara]{Ma1993}
Z.~D. Ma, N.~Kikuchi, and I.~Hagiwara.
\newblock Structural topology and shape optimization for a frequency response
  problem.
\newblock \emph{Comput. Mech.}, 13\penalty0 (3):\penalty0 157--174, 1993.
\newblock \doi{10.1007/bf00370133}.

\bibitem[Michell(1904)]{Michell_1904}
A.~Michell.
\newblock The limits of economy of material in frame-structures.
\newblock \emph{Lond. Edinb. Dubl. Phil. Mag.}, 8\penalty0 (47):\penalty0
  589--597, 1904.
\newblock \doi{10.1080/14786440409463229}.

\bibitem[Nishioka et~al.(2023)Nishioka, Toyoda, Tanaka, and
  Kanno]{nishioka2023}
A.~Nishioka, M.~Toyoda, M.~Tanaka, and Y.~Kanno.
\newblock On a minimization problem of the maximum generalized eigenvalue:
  properties and algorithms, 2023.
\newblock URL \url{https://arxiv.org/abs/2312.01603}.

\bibitem[Ohsaki et~al.(1999)Ohsaki, Fujisawa, Katoh, and Kanno]{Ohsaki1999}
M.~Ohsaki, K.~Fujisawa, N.~Katoh, and Y.~Kanno.
\newblock Semi-definite programming for topology optimization of trusses under
  multiple eigenvalue constraints.
\newblock \emph{Comput. Methods Appl. Mech. Eng.}, 180\penalty0 (1-2):\penalty0
  203--217, 1999.
\newblock \doi{10.1016/s0045-7825(99)00056-0}.

\bibitem[Silva et~al.(2019)Silva, Neves, and Lenzi]{silva_critical_2019}
O.~M. Silva, M.~M. Neves, and A.~Lenzi.
\newblock A critical analysis of using the dynamic compliance as objective
  function in topology optimization of one-material structures considering
  steady-state forced vibration problems.
\newblock \emph{J. Sound Vib.}, 444:\penalty0 1--20, 2019.
\newblock ISSN 0022-460X.
\newblock \doi{10.1016/j.jsv.2018.12.030}.

\bibitem[Tyburec et~al.(2019)Tyburec, Zeman, Nov{\'{a}}k, Lep{\v{s}},
  Plach{\'{y}}, and Poul]{Tyburec2019}
M.~Tyburec, J.~Zeman, J.~Nov{\'{a}}k, M.~Lep{\v{s}}, T.~Plach{\'{y}}, and
  R.~Poul.
\newblock Designing modular 3d printed reinforcement of wound composite hollow
  beams with semidefinite programming.
\newblock \emph{Mater. Des.}, 183:\penalty0 108131, 2019.
\newblock \doi{10.1016/j.matdes.2019.108131}.

\bibitem[Tyburec et~al.(2021)Tyburec, Zeman, Kru{\v{z}}{\'i}k, and
  Henrion]{tyburec_global_2021}
M.~Tyburec, J.~Zeman, M.~Kru{\v{z}}{\'i}k, and D.~Henrion.
\newblock Global optimality in minimum compliance topology optimization of
  frames and shells by moment-sum-of-squares hierarchy.
\newblock \emph{Struct. Multidiscipl. Optim.}, 64\penalty0 (4):\penalty0
  1963--1981, 2021.
\newblock ISSN 1615-1488.
\newblock \doi{10.1007/s00158-021-02957-5}.

\bibitem[Vandenberghe and Boyd(1996)]{Vandenberghe_1996}
L.~Vandenberghe and S.~Boyd.
\newblock Semidefinite programming.
\newblock \emph{{SIAM} Rev.}, 38\penalty0 (1):\penalty0 49--95, 1996.
\newblock \doi{10.1137/1038003}.

\bibitem[Venini(2016)]{Venini2016}
P.~Venini.
\newblock Dynamic compliance optimization: Time vs frequency domain strategies.
\newblock \emph{Comput. Struct.}, 177:\penalty0 12--22, 2016.
\newblock \doi{10.1016/j.compstruc.2016.07.012}.

\bibitem[Wolkowicz et~al.(2000)Wolkowicz, Saigal, Vandenberghe, and
  Hillier]{wolkowicz_handbook_2000}
H.~Wolkowicz, R.~Saigal, L.~Vandenberghe, and F.~S. Hillier, editors.
\newblock \emph{Handbook of semidefinite programming}, volume~27 of
  \emph{International {Series} in {Operations} {Research} \& {Management}
  {Science}}.
\newblock Springer US, 2000.
\newblock ISBN 978-1-4615-4381-7.
\newblock \doi{10.1007/978-1-4615-4381-7}.

\bibitem[Zhang and Kang(2015)]{Zhang2015}
X.~Zhang and Z.~Kang.
\newblock Topology optimization of magnetorheological fluid layers in sandwich
  plates for semi-active vibration control.
\newblock \emph{Smart. Mater. Struct.}, 24\penalty0 (8):\penalty0 085024, 2015.
\newblock \doi{10.1088/0964-1726/24/8/085024}.

\end{thebibliography}

\end{document}